\documentclass[11pt,           
english                       
]{article}

\usepackage[english,strings]{babel}     
\usepackage{amsmath}           
\usepackage{amssymb}
\usepackage{mathrsfs}
\usepackage{mathtools}
\usepackage{stmaryrd}
\usepackage{xkeyval}
\usepackage[utf8]{inputenc}    
\usepackage[T1]{fontenc}       
\usepackage{longtable}         
\usepackage{exscale}           
\usepackage[final]{graphicx}   
\usepackage[sort]{cite}        
\usepackage{array}             
\usepackage{fancyhdr}          
\usepackage[a4paper]{geometry} 
\usepackage{xspace}            
\usepackage{tikz}              
\usepackage{tikz-cd}
\usepackage{nchairx}
\usepackage[expansion=false    
           ]{microtype}        
\usepackage[nottoc]{tocbibind} 
\usepackage[backref=page,      
           final=true,         
           pdfpagelabels       
           ]{hyperref}         
\usepackage[amsmath,thmmarks,hyperref]{ntheorem}
\usepackage{graphicx}
\usepackage{enumitem}
\usepackage{aliascnt}

\usepackage{faktor}  

\graphicspath{{../tikz/pic/}{../tikz/plot/}}

\usetikzlibrary{arrows,intersections,matrix,calc,fadings,decorations.pathreplacing,positioning,patterns,decorations.markings,decorations.pathmorphing}

\newcommand{\myemail}{\texttt{stefan.waldmann@mathematik.uni-wuerzburg.de}}
\newcommand{\myaddress}{Julius Maximilian University of Würzburg \\
     Department of Mathematics \\
     Chair of Mathematics X (Mathematical Physics) \\
     Emil-Fischer-Straße 31 \\
     97074 Würzburg \\
     Germany}

\newcommand{\AuthorOne}{Marvin Dippell}
\newcommand{\AuthorTwo}{Chiara Esposito}
\newcommand{\AuthorThree}{Stefan Waldmann}

\author{\AuthorOne\thanks{\AuthorEmailOne}\\[0.5cm]
\AuthorAddressOne \\[1cm]
    \addtocounter{footnote}{2}
    \AuthorTwo\thanks{\AuthorEmailTwo}\\[0.5cm]
    \AuthorAddressTwo \\[1cm]
    \AuthorThree\thanks{\AuthorEmailThree}\\[0.5cm]
    \AuthorAddressThree
}
\newcommand{\AuthorAddressOne}{\myaddress}
\newcommand{\AuthorAddressTwo}{Dipartimento di Matematica \\
     Universit\`a degli Studi di Salerno \\
     via Giovanni Paolo II, 132 \\
     84084 Fisciano (SA) \\
     Italy
}
\newcommand{\AuthorAddressThree}{\myaddress}

\newcommand{\AuthorEmailOne}{\texttt{marvin.dippell@mathematik.uni-wuerzburg.de}}
\newcommand{\AuthorEmailTwo}{\texttt{chesposito@unisa.it}}
\newcommand{\AuthorEmailThree}{\myemail}

\makeatletter

\newcommand{\Naturals}{\mathbb{N}}

\newcommand{\Integers}{\mathbb{Z}}

\newcommand{\qalgebra}[1]    {{\pmb{\algebra{#1}}}}

\newcommand{\total}{{\scriptscriptstyle\ch@irxscriptfont{total}}}

\newcommand{\strict}{{\scriptscriptstyle\ch@irxscriptfont{strict}}}

\newcommand{\algbimodule}[3]{\deco{}{\algebra{#1}}{\module{#2}}{}{\algebra{#3}}}

\newcommand{\CoisoModTriple}{\categoryname{C_3Mod}}

\newcommand{\CoisoAlgTriple}{\categoryname{C_3Alg}}

\newcommand{\CoisoBimodTriple}{\categoryname{C_3Bimod}}

\newcommand{\CoisoHom}{\categoryname{C_3}\!\operatorname{Hom}}

\newcommand{\CoisoEnd}{\categoryname{C_3}\!\operatorname{End}}

\newcommand{\CoisoAut}{\categoryname{C_3}\!\operatorname{Aut}}

\newcommand{\CoisoDer}{\categoryname{C_3}\!\operatorname{Der}}

\newcommand{\CoisoInnDer}{\categoryname{C_3}\!\operatorname{InnDer}}

\newcommand{\CoChains}{\categoryname{Ch}}

\newcommand{\normalizer}{\operatorname{N}}

\newcommand{\deform}[1]{\boldsymbol{#1}}

\newcommand{\Coeq}{\operatorname{coeq}}

\newcommand{\regimage}{\operatorname{regim}}

\newcommand{\reg}{\ch@irxscriptfont{reg}}

\newcommand{\MCset}{\mathrm{MC}}

\newcommand{\CoisoSetTriple}{\categoryname{C_3Set}}

\newcommand{\CoisoGroupsTriple}{\categoryname{C_3Group}}

\newcommand{\DGLieAlg}{\categoryname{dgLieAlg}}
\newcommand{\CoisoDGLieAlgTriple}{\categoryname{C_3dgLieAlg}}

\renewcommand{\Schouten}[1]{\ch@irxllbbracket #1\ch@irxrrbbracket}

\newcommand{\coisoField}[1]{{\underline{\field{#1}}}}

\newcommand{\Def}{\functor{Def}}

\newcommand{\Hochschild}{\mathrm{HH}}

\providecommand*{\sectionautorefname}{Section}
\providecommand*{\subsectionautorefname}{Subsection}

\newcommand{\Null}{{\scriptscriptstyle0}}
\newcommand{\Wobs}{{\scriptscriptstyle{\ch@irxscriptfont{N}}}}
\newcommand{\Total}{{\scriptscriptstyle\ch@irxscriptfont{tot}}}
\newcommand{\NULL}{0}
\newcommand{\WOBS}{\normalizer}
\newcommand{\TOTAL}{\mathrm{tot}}

\newcommand{\indiscrete}{{\scriptscriptstyle\ch@irxscriptfont{ind}}}
\newcommand{\discrete}{{\scriptscriptstyle\ch@irxscriptfont{dis}}}

\makeatother

\title{Deformation and Hochschild Cohomology of Coisotropic Algebras}

\date{\today}%

\begin{document}

\selectlanguage{english}

%
%

\maketitle
\newpage

%
%

\begin{abstract}
    Coisotropic algebras consist of triples of algebras for which a
    reduction can be defined and unify in a very algebraic fashion
    coisotropic reduction in several settings. In this paper we study
    the theory of (formal) deformation of coisotropic algebras showing
    that deformations are governed by suitable coisotropic DGLAs. We
    define a deformation functor and prove that it commutes with
    reduction. Finally, we study the obstructions to existence and
    uniqueness of coisotropic algebras and present some geometric
    examples.
\end{abstract}

\vspace*{2cm}

%
%

\tableofcontents
\newpage

%
%

\section{Introduction}
\label{sec:Introduction}

Symmetry reduction plays an important role in theoretical classical
mechanics and quantum physics, and its various mathematical
formulations have been studied extensively during the last half
century.  Probably the most well-known reduction procedure of this
kind is the so-called Marsden-Weinstein reduction
\cite{marsden.weinstein:1974a} of a symplectic manifold, which can
also be understood as a special case of coisotropic reduction of a
Poisson manifold.  This standard construction of Poisson geometry
allows to construct a new Poisson manifold out of a given coisotropic
submanifold of a Poisson manifold. The main motivation of such
reduction scheme comes from Dirac's idea \cite{dirac:1950a} of
quantizing the first-class constraints, which are described by
coisotropic submanifolds, and obtaining a quantized version of
coisotropic reduction.

Having this motivation in mind, one can choose deformation
quantization \cite{bayen.et.al:1978a}, see \cite{waldmann:2007a} for a
gentle introduction, to formulate quantization of Poisson geometry.
Here the idea is that a classical mechanical system which is
implemented by a Poisson manifold can equivalently be described by its
Poisson algebra of real-valued functions on it.  The quantized system
corresponds to a (formal) deformation of the commutative algebra of
functions such that the Poisson bracket gets deformed into the
commutator of the possibly non-commutative deformed algebra.  This
procedure relies on a classical principle stating that deformations of
mathematical objects are governed by associated differential graded
Lie algebras (DGLAs). More precisely, formal deformations of an
associative algebra $\algebra{A}$ in the sense of Gerstenhaber
\cite{gerstenhaber:1963a} are given by formal Maurer-Cartan elements
of the associated Hochschild DGLA $C^\bullet(\algebra{A})$, where two
such deformations are considered to be equivalent if they lie in the
same orbit of the action of the canonically associated gauge group.
This leads to the moduli space $\Def$ of formal deformations.  An
important tool to understand formal deformations of associative
algebras is Hochschild cohomology: the second and third Hochschild
cohomology groups contain obstructions to the existence and
equivalence of formal deformations.

In the setting of deformation quantization many versions of phase
space reduction are available, starting with a BRST approach in
\cite{bordemann.herbig.waldmann:2000a} and more general coisotropic
reduction schemes found in e.g. \cite{cattaneo.felder:2007a,
  cattaneo:2004a, cattaneo.felder:2004a,
  bordemann.herbig.pflaum:2007a, bordemann:2005a, bordemann:2004a:pre,
  gutt.waldmann:2010a}. Here reduction is treated in a very algebraic
fashion: the vanishing functions on the coisotropic submanifold are
deformed into a left ideal of the total algebra of all functions and
the reduced algebra is the quotient of the normalizer of this left
ideal modulo the ideal itself.

Recently, we introduced a more algebraic approach to reduction in both
the quantum and classical setting, see
\cite{dippell.esposito.waldmann:2019a}.  In particular, we defined the
notion of coisotropic algebra $\algebra{A}$ consisting of a unital
associative algebra $\algebra{A}_\Total$ together with a unital
subalgebra $\algebra{A}_\Wobs$ and a two-sided ideal
$\algebra{A}_\Null \subseteq \algebra{A}_\Wobs$.  Such coisotropic
algebras allow for a simple reduction procedure, with the reduced
algebra given by
$\algebra{A} = \algebra{A}_\Wobs / \algebra{A}_\Null$.  The eponymous
example is given by a Poisson manifold $M$ together with a coisotropic
submanifold $C$.  Then $(\Cinfty(M), \mathcal{B}_C, \mathcal{J}_C)$,
with $\mathcal{J}_C$ the ideal of functions vanishing on $C$ and
$\mathcal{B}_C$ the Poisson normalizer of $\mathcal{J}_C$, defines a
coisotropic algebra, and its reduced algebra
$\mathcal{B}_C / \mathcal{J}_C$ is isomorphic to the algebra of
functions $\Cinfty(M_\red)$ on the reduced manifold $M_\red$ if the
reduced space is actually smooth. It turns out that one has a
meaningful tensor product leading to a bicategory of bimodules over
coisotropic algebras such that reduction becomes a morphism of
bicategories. Moreover, reduction turns out the be compatible with
classical limits in a nice and general functorial way.  It is
important to notice that this notion recovers other examples coming
from Poisson geometry, e.g. \cite{esposito:2014a} and non commutative
geometry, as \cite{masson:1996a} and \cite{dandrea:2019a}.

Motivated by the significance of coisotropic algebras and their
classical limit, in this paper we develop the corresponding theory of
(formal) deformations. Following the above mentioned classical
principle, we introduce the notion of coisotropic DGLA and we study
formal deformations of the corresponding Maurer-Cartan elements. This
allows us to define a deformation functor and to prove that the
deformation functor commutes with reduction, in the sense that at
least an injective natural transformation exists, see
\autoref{thm:defvsred}. Applying these techniques to the case of the
coisotropic Hochschild complex of a coisotropic algebra we prove that
the existence and uniqueness of formal deformations of coisotropic
algebras are obstructed by its associated coisotropic Hochschild
cohomology, see \autoref{theorem:EquivalenceClassesOfDefs},
\autoref{theorem:Obstructions}.  Moreover, it is shown that the
construction of the coisotropic moduli space of deformations as well
as that of the associated Hochschild cohomology are compatible with
reduction.

The paper is organized as follows: in
\autoref{sec:CoisotropicStructures} some basic coisotropic versions of
classical algebraic structures, such as coisotropic modules,
coistropic algebras and coisotropic complexes, are introduced.  These
notions lead to a definition of a coisotropic DGLA.  In
\autoref{sec:CoisotropicDGLAs} coisotropic DGLAs together with their
coisotropic sets of Maurer-Cartan elements, their associated
coisotropic gauge groups and the formal deformation of coisotropic
Maurer-Cartan elements are considered and the compatibility of these
constructions with reduction is examined.  In the last
\autoref{sec:DeformationTheoryOfCoisotropicAlgebras} we introduce
coisotropic Hochschild cohomology for coisotropic algebras and apply
the results of \autoref{sec:CoisotropicDGLAs} to the case of the
coisotropic Hochschild complex.  Finally, some examples of formal
deformations of coisotropic algebras from geometry are given.

\vspace*{0.5cm}

\noindent
\textbf{Acknowledgements:} It is a pleasure to thank Andreas Kraft for
important remarks on this paper.

%
%

\section{Coisotropic Structures}
\label{sec:CoisotropicStructures}

\subsection{Preliminaries on Coisotropic Modules}
\label{sec:CoisotropicModules}

In the following $\field{k}$ denotes a fixed commutative unital ring,
where we adopt the convention that rings will always be associative.
Let us introduce the fundamental notion of a coisotropic
$\field{k}$-module, which is crucial to all further considerations.
\begin{definition}[Coisotropic $\field{k}$-modules]
    \label{definition:CoisoTriplesModules}%
    Let $\field{k}$ be a commutative unital ring.
    \begin{definitionlist}
    \item \label{item:CoisoTripleModule} A triple
        $\module{E} = (\module{E}_\Total, \module{E}_\Wobs,
        \module{E}_\Null)$ of $\field{k}$-bimodules together with a
        module homomorphism
        $\iota_\module{E} \colon \module{E}_\Wobs \longrightarrow
        \module{E}_\Total$ is called a \emph{coisotropic
          $\field{k}$-module} if
        $\module{E}_\Null \subseteq \module{E}_\Wobs$ is a sub-module.
    \item \label{item:CoisoTripleModuleMorphism} A \emph{morphism
        $\Phi\colon \module{E} \longrightarrow \module{F}$ of
        coisotropic $\field{k}$-modules} is a pair
        $(\Phi_\Total, \Phi_\Wobs)$ of module homomorphisms
        $\Phi_\Total \colon \module{E}_\Total \longrightarrow
        \module{F}_\Total$ and
        $\Phi_\Wobs\colon \module{E}_\Wobs \longrightarrow
        \module{F}_\Wobs$ such that
        $\Phi_\Total \circ \iota_\module{E} = \iota_{\module{F}} \circ
        \Phi_\Wobs$ and
        $\Phi_\Wobs(\module{E}_\Null) \subseteq \module{F}_\Null$.
    \item \label{item:CategoryCoisoTripleMods} The \emph{category of
        coisotropic $\field{k}$-modules} is denoted by
        $\CoisoModTriple_\coisoField{k}$ and the set of morphisms
        between coisotropic $\field{k}$-modules $\module{E}$ and
        $\module{F}$ is denoted by
        $\Hom_\coisoField{k}(\module{E},\module{F})$.
    \end{definitionlist}
\end{definition}

If the underlying ring is clear we will often just use the term
coisotropic module.  We will now collect some useful categorical
constructions for coisotropic modules.
The following statements can be proved by straightforward checks of
the categorical properties, see e.g. \cite{maclane:1998a}.  Let
$\module{E}$, $\module{F}$ be coisotropic modules and let
$\Phi, \Psi \colon \module{E} \to \module{F}$ be morphisms of
coisotropic modules.
\begin{remarklist}
\item[a)] \label{item:Monomorphisms} The morphism $\Phi$ is a
    \emph{monomorphism} if and only if $\Phi_\Total$ and $\Phi_\Wobs$
    are injective module homomorphisms.
\item[b)] \label{item:Epimorphisms} The morphism $\Phi$ is an
    \emph{epimorphism} if and only if $\Phi_\Total$ and $\Phi_\Wobs$
    are surjective module homomorphisms.
\item[c)] \label{item:RegularMono} The morphism $\Phi$ is a \emph{regular
      monomorphism} if and only if it is a monomorphism with
    $\Phi_\Wobs^{-1}(\module{F}_\Null) = \module{E}_\Null$.
\item[d)] \label{item:RegularEpi} The morphism $\Phi$ is a \emph{regular
      epimorphism} if and only if it is an epimorphism with
    $\Phi_\Wobs(\module{E}_\Null) = \module{F}_\Null$.
\end{remarklist}
Observe that the monomorphisms (epimorphisms) in
$\CoisoModTriple_\coisoField{k}$ do in general not agree with regular
monomorphisms (epimorphisms), showing that
$\CoisoModTriple_\coisoField{k}$ is \emph{not} an abelian category,
unlike the usual categories of modules. This will cause some
complications later on.
\begin{remarklist}
\item[e)] \label{Kernel} The \emph{kernel} of $\Phi$ is given by the
    coisotropic module
    \begin{equation}
        \ker(\Phi)
        =
        \big(
        \ker(\Phi_\Total),
        \;
        \ker(\Phi_\Wobs),
        \;
        \ker(\Phi_\Wobs) \cap \module{E}_0
        \big)
    \end{equation}
    with $\iota_{\ker} \colon \ker(\Phi_\Wobs) \to \ker(\Phi_\Total)$
    being the morphism induced by $\iota_{\module{E}}$.
\item[f)] \label{item:Cokernel} The cokernel of $\Phi$ is given by the
    coisotropic module
    \begin{equation}
        \coker(\Phi)
        =
	\big(
        \module{F}_\Total / \image(\Phi_\Total),
        \;
	\module{F}_\Wobs / \image(\Phi_\Wobs),
        \;
	\module{F}_\Null / \image(\Phi_\Wobs)
        \big)
    \end{equation}
    with
    $\iota_{\coker} \colon \module{F}_\Wobs / \image(\Phi_\Wobs) \to
    \module{F}_\Total / \image(\Phi_\Total)$ being the morphism
    induced by $\iota_\module{F}$.
\item[g)] \label{CoisoImage} The coisotropic module
    $\image(\Phi) := \coker(\ker \Phi)$ is given by
    \begin{equation}
        \image(\Phi)
        =
        \big(
        \image(\Phi_\Total),
        \;
        \image(\Phi_\Wobs),
        \;
        \image\big(\Phi_\Wobs\at{\module{E}_\Null}\big)
        \big).
    \end{equation}
    It will be called the \emph{image} of $\Phi$.
\item[h)] \label{CoisoCoimage} The coisotropic module
    $\regimage(\Phi) := \ker(\coker \Phi)$ is given by
    \begin{equation}
        \regimage(\Phi)
        =
        \big(
        \image(\Phi_\Total),
        \;
        \image(\Phi_\Wobs),
        \;
        \image(\Phi_\Wobs) \cap \module{F}_\Null
        \big).
    \end{equation}
    It will be called the \emph{regular image} of $\Phi$.
\end{remarklist}
In the case of abelian categories, there is a canonical image
factorization as $\coker (\ker \Phi ) \simeq \ker(\coker \Phi)$ for
every morphism.  This is not the case in the non-abelian category
$\CoisoModTriple_\coisoField{k}$, leading to two different
factorization systems.  Using the image every morphism of coisotropic
modules can be factorized into a regular epimorphism and a
monomorphism while using the regular image allows for a factorization
into an epimorphism and a regular monomorphism.
\begin{remarklist}
\item[i)]\label{item:Coequalizer} The \emph{coequalizer} of $\Phi$ and
    $\Psi$ is given by the coisotropic module
    \begin{equation}
        \label{eq:Coequalizer}
	\Coeq(\Phi,\Psi)
	=
        \big(
        \Coeq(\Phi_\Total, \Psi_\Total),
        \;
        \Coeq(\Phi_\Wobs, \Psi_\Wobs),
        \;
        q_\Wobs(\module{F}_\Null)
        \big)
    \end{equation}
    with
    $q_\Wobs \colon \module{F}_\Wobs \to \Coeq(\Phi_\Wobs,
    \Psi_\Wobs)$ being the coequalizer morphism of
    $\Phi_\Wobs, \Psi_\Wobs$ and
    $\iota_{\Coeq} \colon \Coeq(\Phi_\Wobs, \Psi_\Wobs) \to
    \Coeq(\Phi_\Total, \Psi_\Total)$ being the morphism induced by the
    morphisms $\Phi_\Wobs \circ \iota_\module{F}$ and
    $\Psi_\Wobs \circ \iota_\module{F}$.
\item[j)] \label{Quotient} Let $\module{E}' \subseteq \module{E}$ be a
    coisotropic submodule, i.e.
    $\module{E}'_\Total \subseteq \module{E}_\Total$,
    $\module{E}'_\Wobs \subseteq \module{E}_\Wobs$ and
    $\module{E}'_\Null \subseteq \module{E}_\Null$, and denote by
    $\mathrm{i} \colon \module{E}' \to \module{E}$ the inclusion
    morphism.  The \emph{quotient} of $\module{E}$ by $\module{E}'$ is
    then the coequalizer of $\mathrm{i}$ and the zero map.  More
    explicitly, we get
    \begin{equation}
        \label{eq:QuotientsCoisotropicModules}
	\module{E} / \module{E}'
	=
        \big(
        \module{E}_\Total / \module{E}'_\Total,
        \;
	\module{E}_\Wobs / \module{E}'_\Wobs,
        \;
	\module{E}_\Null / \module{E}'_\Wobs
        \big).
    \end{equation}
\item[k)] \label{item:Coproduct} The \emph{coproduct} of $\module{E}$ and
    $\module{F}$ is given by
    \begin{equation}
        \module{E} \oplus \module{F}
        =
        \big(
        \module{E}_\Total \oplus \module{F}_\Total,
        \;
        \module{E}_\Wobs \oplus \module{F}_\Wobs,
        \;
        \module{E}_\Null \oplus \module{F}_\Null
        \big)
    \end{equation}
    with $\iota_\oplus = \iota_\module{E} + \iota_\module{F}$. It is
    called the \emph{direct sum} of $\module{E}$ and $\module{F}$.  It
    should be clear that also infinite direct sums can be defined this
    way.  Finite direct sums of coisotropic modules can be shown to be
    biproducts for the category $\CoisoModTriple_\coisoField{k}$.  In
    particular, it is clear that also products exist.
\end{remarklist}

A fundamental notion in this setting is the tensor product of
coisotropic modules.  This is an additional piece of information and
is not fixed solely from the definition of the category
$\CoisoModTriple_\coisoField{k}$.
\begin{definition}[Tensor product]
    \label{def:TensorProductCoisotropicModules}%
    Let $\module{E}, \module{F} \in \CoisoModTriple_\coisoField{k}$ be
    coisotropic modules.  The coisotropic module
    \begin{equation}
	\module{E} \tensor \module{F}
	=
        \big(
        \module{E}_\Total \tensor \module{F}_\Total,
        \;
	\module{E}_\Wobs \tensor \module{F}_\Wobs,
        \;
	\module{E} _\Wobs \tensor \module{F}_\Null
        +
        \module{E}_\Null \tensor \module{F}_\Wobs
        \big)
    \end{equation}
    with $\iota_{\tensor} = \iota_\module{E} \tensor \iota_\module{F}$
    is called the \emph{tensor product} of $\module{E}$ and
    $\module{F}$.
\end{definition}
\begin{remark}
    \label{remark:TensorProduct}%
    Let $\module{E}, \module{F} \in \CoisoModTriple_\coisoField{k}$ be
    coisotropic modules.
    \begin{remarklist}
    \item \label{item:TensorProductIsCoiso} The triple
        $\module{E} \tensor \module{F}$ is indeed a coisotropic
        $\field{k}$-module.  In particular,
        $(\module{E} \tensor \module{F})_\Null = \module{E}_\Wobs
        \tensor \module{F}_\Null + \module{E}_\Null \tensor
        \module{F}_\Wobs$ is the submodule of
        $\module{E}_\Wobs \tensor \module{F}_\Wobs$ generated by
        elements of the form $x \tensor y$ with
        $x \in \module{E}_\Wobs$, $y \in \module{F}_\Null$ or
        $y \in \module{E}_\Null$, $y \in \module{F}_\Wobs$.
    \item \label{item:iotaNotInjective} The reason we did not insist
        on $\iota$ being injective in
        \autoref{definition:CoisoTriplesModules} is that the
        injectivity of $\iota_{\tensor}$ may be spoiled by torsion
        effects.  Nevertheless, in many examples this will be the
        case.
    \end{remarklist}
\end{remark}

This definition of tensor product allows us to construct a functor
$\tensor \colon \CoisoModTriple_\coisoField{k} \times
\CoisoModTriple_\coisoField{k} \to \CoisoModTriple_\coisoField{k}$,
which together with the coisotropic module
$\coisoField{k} = (\field{k},\field{k},0)$ as unit object turns
$\CoisoModTriple_\coisoField{k}$ into a (weak) monoidal category, see
e.g. \cite{etingof.gelaki.nikshych.ostrik:2015a}.

\begin{remarklist}
\item[l)] \label{item:SymmetricMonoidal} The monoidal category
    $\CoisoModTriple_\coisoField{k}$ is a symmetric monoidal category
    with symmetry
    $\tau \colon \module{E} \tensor \module{F} \to \module{F} \tensor
    \module{E}$ given by
    $\tau_{\Total/\Wobs}(x \tensor y) = y \tensor x$.
\item[m)] \label{InternalHom} The internal Hom of $\module{E}$ and
    $\module{F}$ is given by the coisotropic module
    \begin{equation}
        \begin{split}
            \CoisoHom_\coisoField{k}(\module{E},\module{F})_\Total
            &:=
            \Hom_\field{k}(\module{E}_\Total, \module{F}_\Total), \\
            \CoisoHom_\coisoField{k}(\module{E},\module{F})_\Wobs
            &:=
            \Hom_\coisoField{k}(\module{E},\module{F}), \\
            \CoisoHom_\coisoField{k}(\module{E},\module{F})_\Null
            &:=
            \left\{
                (\Phi_\Total, \Phi_\Wobs)
                \in
                \Hom_\coisoField{k}(\module{E},\module{F})
                \mid
                \Phi_\Wobs(\module{E}_\Wobs) \subseteq \module{F}_\Null
            \right\},
        \end{split}
    \end{equation}
    where
    $\iota \colon \Hom_\coisoField{k}(\module{E},\module{F}) \to
    \Hom_\field{k}(\module{E}_\Total, \module{F}_\Total)$ is the
    projection onto the first component.  We will denote the
    coisotropic module of endomorphisms by
    $\CoisoEnd_{\coisoField{k}}(\module{E}) :=
    \CoisoHom_\coisoField{k}(\module{E},\module{E})$.  Similarly, the
    coisotropic automorphisms are denoted by
    $\CoisoAut_{\coisoField{k}}$.  This internal Hom is in fact right
    adjoint to the tensor product. More precisely, we have
    $\argument \tensor \module{E}$ is left adjoint to
    $\CoisoHom(\module{E},\argument)$, showing that
    $\CoisoModTriple_\coisoField{k}$ is in fact closed monoidal.  From
    this follows in particular that for every $x \in \module{E}_\Wobs$
    and $\Phi \colon \module{E} \tensor \module{F} \to \module{G}$ we
    get a coisotropic coevaluation morphism of modules
    $\Phi(x, \argument) \colon \module{F} \to \module{G}$.
\end{remarklist}

Let us stress that $\Hom_\coisoField{k}(\module{E}, \module{F})$ only
denotes the set of coisotropic morphisms and
$\CoisoHom_\coisoField{k}(\module{E}, \module{F})$ denotes the full
coisotropic module of morphisms.  The definition of coisotropic
modules allows us to reinterpret several (geometric) reduction
procedures in a completely algebraic fashion, as stated in the
following straightforward proposition.
\begin{proposition}[Reduction]
    \label{prop:Reduction}%
    Mapping a coisotropic module $\module{E}$ to the quotient
    $\module{E}_\red = \module{E}_\Wobs / \module{E}_\Null$ and
    morphisms of coisotropic modules to the induced morphisms yields a
    monoidal functor
    \begin{equation}
        \red \colon \CoisoModTriple_\coisoField{k} \to \Modules_\field{k},
    \end{equation}
    where the category $\Modules_\field{k}$ of $\field{k}$-bimodules
    is equipped with the usual tensor product.
\end{proposition}
\begin{remark}
    \label{rmk:ReductionOfInternalHom}%
    Since the internal Hom
    $\CoisoHom_\coisoField{k}(\module{E},\module{F})$ is a coisotropic
    module itself we can apply the reduction functor $\red$ to it.
    There is a canonical $\field{k}$-module morphism from
    $\CoisoHom_\coisoField{k}(\module{E},\module{F})_\red$ to
    $\Hom_\field{k}(\module{E}_\red, \module{F}_\red)$ given by
    mapping $[(\Phi_\Total,\Phi_\Wobs)]$ to the map $[\Phi_\Wobs]$
    induced by $\Phi_\Wobs$ on the quotient
    $\module{E}_\red = \module{E}_\Wobs / \module{E}_\Null$.  Note
    that this morphism is injective.  Therefore, we can view
    $\CoisoHom_\coisoField{k}(\module{E},\module{F})_\red$ as the
    submodule of $\Hom_\field{k}(\module{E}_\red,\module{F}_\red)$
    consisting of morphism that allow for a extension to the
    $\TOTAL$-components of $\module{E}$ and $\module{F}$.
\end{remark}

%
%

\subsection{Coisotropic Algebras \& Derivations}
\label{sec:CoisotropicAlgebras}

Consider again the prototypical example of a coisotropic submanifold
$C \hookrightarrow M$ of a given Poisson manifold $(M,\pi)$.  Then the
coisotropic module $(\Cinfty(M), \mathcal{B}_C,\mathcal{J}_C)$
obviously carries more structure than a mere coisotropic module.  In
particular, $\Cinfty(M)$ is an associative algebra with
$\mathcal{B}_C \subseteq \Cinfty(M)$ a subalgebra and
$\mathcal{J}_C \subseteq \mathcal{B}_C$ a two-sided ideal.  This is
now captured by the following definition of a coisotropic algebra.
\begin{definition}[Coisotropic algebra]
    \label{def:CoisotropicAlgebras}%
    Let $\field{k}$ be a commutative unital ring.
    \begin{definitionlist}
    \item \label{item:CoisotropicAlgebra} A \emph{coisotropic algebra}
        over $\mathbb{k}$ is a triple
        $\algebra{A} = (\algebra{A}_\Total, \algebra{A}_\Wobs,
        \algebra{A}_\Null)$ consisting of unital associative algebras
        $\algebra{A}_\Total$, $\algebra{A}_\Wobs$ and a two-sided
        ideal $\algebra{A}_\Null \subseteq \algebra{A}_\Wobs$ together
        with a unital algebra homomorphism
        $\iota \colon \algebra{A}_\Wobs \to \algebra{A}_\Total$.
    \item \label{item:MorphismCoisoAlgebra} A \emph{morphism
        $\Phi \colon \algebra{A} \to \algebra{B}$ of coisotropic
        algebras} is given by a pair of unital algebra homomorphisms
        $\Phi_\Total \colon \algebra{A}_\Total \to \algebra{B}_\Total$
        and
        $\Phi_\Wobs \colon \algebra{A}_\Wobs \to \algebra{B}_\Wobs$
        such that
        $\iota_\algebra{B} \circ \Phi_\Wobs = \Phi_\Total \circ
        \iota_\algebra{A}$ and
        $\Phi_\Wobs(\algebra{A}_\Null) \subseteq \algebra{B}_\Null$.
    \item \label{item:CategoryCoisoAlg} The \emph{category of coisotropic
        $\field{k}$-algebras} is denoted by
        $\CoisoAlgTriple_\coisoField{k}$.
    \end{definitionlist}
\end{definition}
Coisotropic algebras can also be understood as internal algebras in
the monoidal category $\CoisoModTriple_\coisoField{k}$.  Here the
particular definition of the tensor product of coisotropic modules,
see \autoref{def:TensorProductCoisotropicModules}, is crucial in order
to realize $\algebra{A}_\Null$ as a two-sided ideal in
$\algebra{A}_\Wobs$.  Note that the definition of a coisotropic
algebra as provided above generalizes the one given in
\cite{dippell.esposito.waldmann:2019a} slightly in that we do not
assume $\iota \colon \algebra{A}_\Wobs \to \algebra{A}_\Total$ to be
injective and $\algebra{A}_\Null$ needs not to be a left-ideal in
$\algebra{A}_\Total$. Nevertheless, in most of our applications these
additional features (requirements in
\cite{dippell.esposito.waldmann:2019a}) will be satisfied.

\begin{remark}
    \label{remark:ReducedAlgebra}%
    Since $\algebra{A}_\Null \subseteq \algebra{A}_\Wobs$ is a
    two-sided ideal by definition, we can construct a reduced algebra
    $\algebra{A}_\red = \algebra{A}_\Wobs / \algebra{A}_\Null$ similar
    to \autoref{prop:Reduction}. This yields a functor
    $\red \colon \CoisoAlgTriple_{\coisoField{k}} \to
    \Algebras_\field{k}$.
\end{remark}
\begin{example}
    \label{ex:CoisotropicAlgebras}%
    \begin{examplelist}
    \item \label{ex:CoisotropicAlgebrasFoliation} Let $C \subseteq M$
        be a submanifold and let $\mathcal{F}$ be a foliation on $C$.
        Then
        $\algebra{A} = (\Cinfty(M),
        \Cinfty(M)^\mathcal{F},\mathcal{J}_C)$, with
        $\Cinfty(M)^\mathcal{F}$ the set of functions on $M$ constant
        along the leaves on $C$ and $\mathcal{J}_C$ the vanishing
        ideal of $C$, is a coisotropic algebra.  As soon as the leaf
        space $C/\mathcal{F}$ carries a canonical manifold structure
        we have $\algebra{A}_\red \simeq \Cinfty(C/\mathcal{F})$.
    \item \label{ex:CoisotropicAlgebrasPoisson} Let $(M,\pi)$ be a
        Poisson manifold together with a coisotropic submanifold
        $C \hookrightarrow M$.  Then
        $\algebra{A} = (\Cinfty(M), \mathcal{B}_C, \mathcal{J}_C)$ is
        a coisotropic algebra and
        $\algebra{A}_\red \cong \mathcal{B}_C / \mathcal{J}_C$ turns
        out to be even a Poisson algebra.
    \end{examplelist}
\end{example}

On one hand, from an algebraic point of view, representations are
important in the study of algebraic structures. On the other hand, by
the famous Serre-Swan theorem, vector bundles over manifolds can
equivalently be understood as finitely generated projective modules
over the algebra of functions on the manifold. This justifies to take
a closer look at modules in our context as well. The following gives a
useful notion of (bi-)module over coisotropic algebras:
\begin{definition}[Bimodules over coisotropic algebras]
    \label{definition:CoisoBimodules}%
    Let $\algebra{A}, \algebra{B} \in \CoisoAlgTriple_\coisoField{k}$
    be coisotropic algebras.
    \begin{definitionlist}
    \item \label{item:CoisoTripleBimodule} A triple
        $\module{E} = (\module{E}_\Total, \module{E}_\Wobs,
        \module{E}_\Null)$ consisting of a
        $(\algebra{B}_\Total, \algebra{A}_\Total)$-bimodule
        $\module{E}_\Total$ and
        $(\algebra{B}_\Wobs, \algebra{A}_\Wobs)$-bimodules
        $\module{E}_\Wobs$ and $\module{E}_\Null$ together with a
        bimodule morphism
        $\iota_\module{E} \colon \module{E}_\Wobs \longrightarrow
        \module{E}_\Total$ along the morphisms
        $\iota_\algebra{B} \colon \algebra{B}_\Wobs \to
        \algebra{B}_\Total$ and
        $\iota_\algebra{A} \colon \algebra{A}_\Wobs \to
        \algebra{A}_\Total$ is called a \emph{coisotropic
          $(\algebra{B},\algebra{A})$-bimodule} if
        $\module{E}_\Null \subseteq \module{E}_\Wobs$ is a
        sub-bimodule such that
        \begin{equation}
            \label{eq:BimoduleTripleCondition}
            \algebra{B}_\Null \cdot \module{E}_\Wobs
            \subseteq
            \module{E}_\Null
            \quad
            \textrm{and}
            \quad
            \module{E}_\Wobs \cdot \algebra{A}_\Null
            \subseteq
            \module{E}_\Null.
        \end{equation}
    \item \label{item:CoisoTripleBimoduleMorphism} A \emph{morphism
        $\Phi\colon \module{E} \longrightarrow \tilde{\module{E}}$
        between coisotropic $(\algebra{B}, \algebra{A})$-bimodules} is
        a pair $(\Phi_\Total, \Phi_\Wobs)$ of a
        $(\algebra{B}_\Total, \algebra{A}_\Total)$-bimodule morphism
        $\Phi_\Total \colon \module{E}_\Total \longrightarrow
        \tilde{\module{E}}_\Total$ and a
        $(\algebra{B}_\Wobs,\algebra{A}_\Wobs)$-bimodule morphism
        $\Phi\colon \module{E}_\Wobs \longrightarrow
        \tilde{\module{E}}_\Wobs$ such that
        $\Phi_\Total \circ \iota_\module{E} =
        \iota_{\tilde{\module{E}}} \circ \Phi_\Wobs$ and
        $\Phi_\Wobs(\module{E}_\Null) \subseteq
        \tilde{\module{E}}_\Null$.
    \item \label{item:CategoryCoisoTripleBimods} The \emph{category of
        coisotropic $(\algebra{B}, \algebra{A})$-bimodules} is denoted
        by $\CoisoBimodTriple(\algebra{B}, \algebra{A})$.
    \end{definitionlist}
\end{definition}
Note that a coisotropic $(\algebra{B},\algebra{A})$-bimodule
$\module{E}$ can also be defined as a coisotropic $\field{k}$-module
together with morphisms
$\lambda \colon \algebra{B} \tensor \module{E} \to \module{E}$ and
$\rho \colon \module{E} \tensor \algebra{A} \to \module{E}$ of
coisotropic modules implementing the module structure.  The tensor
product of coisotropic $\field{k}$-modules as defined in
\autoref{def:TensorProductCoisotropicModules} can be extended to
bimodules over coisotropic algebras in the following way.
\begin{lemma}
    \label{lemma:TensorProductBimodules}%
    Let $\algebra{A}$, $\algebra{B}$ and $\algebra{C}$ be coisotropic
    algebras and let
    $\module{F} \in \CoisoBimodTriple{(\algebra{C}, \algebra{B})}$ as
    well as
    $\module{E} \in \CoisoBimodTriple{(\algebra{B},\algebra{A})}$ be
    corresponding bimodules. Then
    $\algbimodule{C}{F}{B} \tensor[\algebra{B}] \algbimodule{B}{E}{A}$
    given by the components
    \begin{align}
	\label{eq:TotalComponentTensorProduct}
	\left(
	\algbimodule{C}{F}{B}
	\tensor[\algebra{B}]
	\algbimodule{B}{E}{A}
	\right)_{\Total}
	&=
	\module{F}_{\Total}
	\tensor[\algebra{B}_{\Total}]
	\module{E}_{\Total}, \\
	\label{eq:WobsComponentTensorProduct}
	\left(
	\algbimodule{C}{F}{B}
	\tensor[\algebra{B}]
	\algbimodule{B}{E}{A}
	\right)_{\Wobs}
	&=
	\module{F}_{\Wobs}
	\tensor[\algebra{B}_{\Wobs}]
	\module{E}_{\Wobs}, \\
	\label{eq:NullComponentTensorProduct}
	\left(
	\algbimodule{C}{F}{B}
	\tensor[\algebra{B}]
	\algbimodule{B}{E}{A}
	\right)_{\Null}
	&=
	\module{F}_{\Wobs}
	\tensor[\algebra{B}_\Wobs]
	\module{E}_{\Null}
	+
	\module{F}_{\Null}
	\tensor[\algebra{B}_\Wobs]
	\module{E}_\Wobs
    \end{align}
    with $\iota_{\tensor} = \iota_\module{F} \tensor \iota_\module{E}$
    is a $(\algebra{C}, \algebra{A})$-bimodule.
\end{lemma}
Coisotropic $\field{k}$-modules can be understood as bimodules for the
coisotropic algebra
$\coisoField{k} = \left( \field{k}, \field{k}, 0 \right)$, explaining
our notation for the category $\CoisoModTriple_\coisoField{k}$ of
coisotropic $\field{k}$-modules.

\begin{example}
    \label{example:CovariantDerivativeFoliation}%
    Let $\iota \colon C \subseteq M$ be a submanifold and
    $D\subseteq TC$ an integrable distribution on $C$.  Let moreover
    $E_\Total \to M$ be vector bundle over $M$, $E_\Wobs \to M$ a
    subbundle of $E_\Total$ and $E_\Null \to C$ a subbundle of
    $\iota^\#E_\Wobs$.  Moreover, let $\nabla$ be a flat partial
    $D$-connection on $\iota^\#E_\Wobs$.  Then setting
    \begin{gather}
        \label{eq:EtotVB}
        \module{E}_\Total
        =
        \Secinfty(E_\Total),
        \\
        \module{E}_\Wobs
        =
        \left\{
            s \in \Secinfty(E_\Wobs)
            \Bigm|
            \nabla_X \iota^\#s = 0
            \textrm{ for all }
            X \in \Secinfty(D)
        \right\},
        \\
        \shortintertext{and}
        \module{E}_\Null
        =
        \left\{
            s \in \Secinfty(E_\Wobs)
            \Bigm|
            \nabla_X \iota^\#s = 0
            \textrm{ for all }
            X \in \Secinfty(D)
            \text{ and }
            \iota^\#s \in \Secinfty(E_\Null)
        \right\}
    \end{gather}
    defines a coisotropic $\algebra{A}$-module $\module{E}$ for
    $\algebra{A} = (\Cinfty(M), \Cinfty(M)^\mathcal{F},
    \mathcal{J}_C)$ as in \autoref{ex:CoisotropicAlgebras},
    \ref{ex:CoisotropicAlgebrasFoliation} with $\mathcal{F}$ the
    foliation induced by $D$.  Note that the construction of
    $\module{E}_\Wobs$ strongly depends on the choice of the covariant
    derivative.  Coisotropic modules of this form are important in a
    coisotropic version of the Serre-Swan theorem, see
    \cite{dippell.menke.waldmann:2020a}.
\end{example}

Coisotropic algebras together with coisotropic bimodules, their
morphisms and their tensor product as above can be arranged in a
bicategory structure.  Mapping a coisotropic algebra $\algebra{A}$ to
its reduced algebra
$\algebra{A}_\red = \algebra{A}_\Total / \algebra{A}_\Wobs$ and a
coisotropic $(\algebra{A},\algebra{B})$-bimodule $\module{E}$ to the
$(\algebra{A}_\red, \algebra{B}_\red)$-bimodule
$\module{E}_\red = \module{E}_\Wobs / \module{E}_\Null$ defines a
functor of bicategories, see \cite{dippell.esposito.waldmann:2019a}.

From a geometric perspective the tangent bundle of a given manifold
corresponds to the derivations of the algebra of functions on that
manifold by taking sections. In order to give a definition of a
derivation of a coisotropic algebra we rephrase the classical
definition in an element-independent way.
\begin{definition}[Derivation]
    \label{def:CoisotropicDerivation}%
    Let $\module{M} \in \CoisoBimodTriple(\algebra{A},\algebra{A})$ be
    an $\algebra{A}$-bimodule.  A \emph{derivation} with values in
    $\module{M}$ is a morphism
    $D \colon \algebra{A} \longrightarrow \module{M}$ of coisotropic
    $\field{k}$-modules such that
    \begin{equation}
        \label{eq:Derivation}
	D \circ \mu_\algebra{A}
        =
        \lambda \circ (\id \tensor D) + \rho \circ (D \tensor \id)
    \end{equation}
    holds, where $\rho$ and $\lambda$ denote the right and left
    $\algebra{A}$-multiplications of $\module{M}$, respectively.  The
    set of derivations will be denoted by
    $\Der(\algebra{A},\module{M})$. If $\module{M} = \algebra{A}$ we
    write $\Der(\algebra{A})$.
\end{definition}

We can arrange the coisotropic derivations as a coisotropic submodule
of the internal homomorphism
$\CoisoHom_\coisoField{k}(\algebra{A},\module{M})$ as follows.
\begin{proposition}
    \label{proposition:CosioDer}%
    Let $\module{M} \in \CoisoBimodTriple(\algebra{A},\algebra{A})$ be
    a coisotropic $\algebra{A}$-bimodule.  Then
    \begin{align}
	\CoisoDer(\algebra{A},\module{M})_\Total
	&:=
        \Der(\algebra{A}_\Total, \module{M}_\Total),
        \\
	\CoisoDer(\algebra{A},\module{M})_\Wobs
	&:=
        \big\{
        (D_\Total,D_\Wobs) \in \Hom_\coisoField{k}(\algebra{A},\module{M})
	\; \big| \;
        D_\Total \in \Der(\algebra{A}_\Total,\module{M}_\Total),
        D_\Wobs \in \Der(\algebra{A}_\Wobs, \module{M}_\Wobs)
        \big\},
        \\
	\CoisoDer(\algebra{A},\module{M})_\Null
	&:=
        \big\{
        (D_\Total, D_\Wobs ) \in \Der(\algebra{A},\module{M})_\Wobs
	\; \big| \;
        D_\Wobs(\algebra{A}_\Wobs) \subseteq \module{M}_\Null
        \big\}
    \end{align}
    defines a coisotropic $\field{k}$-module
    $\CoisoDer(\algebra{A},\module{M})$.
\end{proposition}
One needs to be careful with the notation here since
$\Der(\algebra{A})$ has different meanings depending whether
$\algebra{A}$ is a coisotropic or a classical algebra.  Note also that
$\CoisoDer(\algebra{A},\module{M})_\Wobs =
\Der(\algebra{A},\module{M})$ is just the set of derivations of a
coisotropic algebra $\algebra{A}$ with values in the coisotropic
module $\module{M}$ as given in \autoref{def:CoisotropicDerivation}.
The coisotropic $\field{k}$-module of derivations on $\algebra{A}$
with values in $\algebra{A}$ is denoted by $\CoisoDer(\algebra{A})$.

As for usual algebras the derivations turn out to be a bimodule if the
algebra is commutative:
\begin{proposition}[$\algebra{A}$-module of derivations]
    \label{proposition:DerivationsAsAModule}%
    Let $\algebra{A} \in \CoisoAlgTriple_{\field{k}}$ be a commutative
    coisotropic algebra.  Then $\CoisoDer(\algebra{A})$ is a
    coisotropic $\algebra{A}$-bimodule.
\end{proposition}

Every $(D_\Total, D_\Wobs) \in \Der(\algebra{A})_\Wobs$ defines a
derivation on
$\algebra{A}_\red = \algebra{A}_\Wobs / \algebra{A}_\Null$ since the
condition $D_\Wobs(\algebra{A}_\Null) \subseteq \algebra{A}_\Null$ is
automatically satisfied.  Hence we have a $\field{k}$-linear map
$\Der(\algebra{A})_\Wobs \to \Der(\algebra{A}_\red)$.  The kernel of
this linear map is exactly given by $\CoisoDer(\algebra{A})_\Null$,
thus there exists an injective module homomorphism
\begin{equation}
    \label{eq:ReducedDerivation}
    \CoisoDer(\algebra{A})_\red \hookrightarrow \Der(\algebra{A}_\red).
\end{equation}
This is simply the restriction of the canonical injective morphism
$\CoisoHom_\coisoField{k}(\algebra{A},\algebra{A})_\red \to
\Hom_\field{k}(\algebra{A}_\red,\algebra{A}_\red)$
from \autoref{rmk:ReductionOfInternalHom}
to the submodule $\CoisoDer(\algebra{A})$.
\begin{example}
    Our notion of a coisotropic algebra generalizes and unifies
    previous notions used in noncommutative geometry referring to
    features of the derivations:
    \begin{examplelist}
    \item \label{item:SubmanifoldAlgebra} A \emph{submanifold algebra}
        in the sense of \cite{masson:1996a} and \cite{dandrea:2019a}
        can equivalently be described as a coisotropic algebra
        $\algebra{A}$ with $\algebra{A}_\Total = \algebra{A}_\Wobs$
        such that the canonical module morphism
        \eqref{eq:ReducedDerivation} is an isomorphism.
    \item \label{item:QuotientManifoldAlgebra} A \emph{quotient
          manifold algebra} in the sense of \cite{masson:1996a} can
        equivalently be described as a coisotropic algebra
        $\algebra{A}$ with
        $\algebra{A}_\Wobs \subseteq \algebra{A}_\Total$ a subalgebra
        and $\algebra{A}_\Null = 0$ such that
        $\Center(\algebra{A}_\red) \simeq \Center(\algebra{A})_\red$,
        $\Der(\algebra{A}_\red) \simeq \CoisoDer(\algebra{A})_\red$
        via \eqref{eq:ReducedDerivation} and
        $\algebra{A}_\Wobs = \{ a \in \algebra{A}_\Total \mid \textrm{
          for all } (D_\Total,D_\Wobs) \in
        \CoisoDer(\algebra{A})_\Null \textrm{ one has } D_\Total(a) =
        0 \}$ holds.  Here $\Center(\algebra{A})$ denotes the
        coisotropic center of the coisotropic algebra $\algebra{A}$,
        see \autoref{prop:ZerothFirstHochschild},
        \ref{prop:ZerothFirstHochschild_1} for the definition.
    \end{examplelist}
\end{example}

We can also define inner derivations by requiring the existence of
appropriate elements in each component.
\begin{proposition}
    \label{proposition:InnerDerivations}%
    Let $\algebra{A} \in \CoisoAlgTriple_{\coisoField{k}}$ be a
    coisotropic algebra.  Then
    \begin{align}
	\CoisoInnDer(\algebra{A})_\Total
	&:=
        \InnDer(\algebra{A}_\Total), \\
	\CoisoInnDer(\algebra{A})_\Wobs
	&:=
        \big\{
        (D_\Total,D_\Wobs) \in \CoisoDer(\algebra{A})_\Wobs
	\bigm|
        \exists a \in \algebra{A}_\Wobs : D_\Wobs = [\argument, a]_\Wobs
	\textrm{ and }
        D_\Total = [\argument, \iota_\algebra{A}(a)]_\Total
        \big\}, \\
	\CoisoInnDer(\algebra{A})_\Null
	&:=
        \big\{
        (D_\Total,D_\Wobs) \in \CoisoDer(\algebra{A})_\Null
	\bigm|
        \exists a \in \algebra{A}_\Wobs : D_\Wobs = [\argument, a]_\Wobs
	\textrm{ and }
        D_\Total = [\argument, \iota_\algebra{A}(a)]_\Total
        \big\}
    \end{align}
    defines a coisotropic $\field{k}$-module
    $\CoisoInnDer(\algebra{A})$.
\end{proposition}

%
%

\subsection{Coisotropic Homological Algebra}
\label{sec:CoisotropicHomologicalAlgebra}

We collect some definitions and statements about (cochain) complexes
of coisotropic modules.  Most of this can be done as in every abelian
category.  But since $\CoisoModTriple_\coisoField{k}$ is not abelian
we have to be careful when defining coisotropic cohomology, since we
have two different notions of images, see
\autoref{sec:CoisotropicModules}, \ref{CoisoImage} and
\ref{CoisoCoimage}.

\begin{definition}[Graded coisotropic module]
    \label{definition:GradedStuff}%
    Let $\field{k}$ be a commutative unital ring.
    \begin{definitionlist}
    \item \label{item:GradedCoisoModule} A \emph{($\Integers$-)graded
          coisotropic module} is a $\Integers$-indexed family
        $\{\module{M}^i\}_{i\in \Integers}$ of coisotropic modules
        $\module{M}^i \in \CoisoModTriple_\coisoField{k}$.
    \item \label{item:MorphismGradedCoiso} A \emph{morphism
        $\{\module{M}^i\}_{i \in \Integers} \longrightarrow
        \{\module{N}^i\}_{i \in \Integers}$ of graded coisotropic
        modules} is given by a $\Integers$-indexed family
        $\{\Phi^i\}_{i \in \Integers}$ of morphisms
        $\Phi^i \colon \module{M}^i \longrightarrow \module{N}^i$.
    \item \label{item:GradedCoisoCat} We denote the \emph{category of
          graded coisotropic modules} by
        $\CoisoModTriple_\coisoField{k}^\bullet$.
    \end{definitionlist}
\end{definition}

We combine the indexed family of a graded coisotropic module into a
single coisotropic module
$\module{M}^\bullet = \bigoplus_{i \in \Integers} \module{M}^i$.
Conversely, if a given coisotropic module $\module{M}$ decomposes into
a direct sum indexed by $\Integers$ we write $\module{M}^\bullet$ if we want to
emphasize the graded structure. This way, every
coisotropic module can be viewed as a graded coisotropic module by
placing it at $i=0$ with all other components being trivial.

A more flexible notion of morphism between graded coisotropic modules
is given by a morphism of degree $k$, i.e.  a family
$\Phi^i \colon \module{M}^i \longrightarrow \module{N}^{i+k}$.

We will use the usual tensor product
\begin{equation}
    \label{eq:GradedTensorProduct}
    \module{M} \tensor \module{N}
    =
    \bigoplus_{n \in \mathbb{Z}}
    \Big(\bigoplus_{k + \ell = n} \module{M}^k \tensor \module{N}^\ell \Big),
\end{equation}
and the symmetry with the usual Koszul signs.

\begin{definition}[Coisotropic complex]
    \label{definition:CoisoComplex}%
    Let $\field{k}$ be a commutative unital ring.
    \begin{definitionlist}
    \item \label{item:CoisoChainComplex} A \emph{coisotropic (cochain)
          complex} is a graded coisotropic module $\module{M}^\bullet$
        together with a degree $+1$ map
        $\delta^\bullet \colon \module{M}^\bullet \longrightarrow
        \module{M}^{\bullet+1}$ such that $\delta \circ \delta = 0$.
    \item \label{item:CoisoChainMorphism} A \emph{morphism of coisotropic
        complexes} is a morphism
        $\Phi \colon \module{M}^\bullet \longrightarrow
        \module{N}^\bullet$ such that
        $\Phi \circ \delta_\module{M} = \delta_\module{N} \circ \Phi$.
    \item \label{item:CoisoComplexCat} The \emph{category of coisotropic
        complexes} is denoted by
        $\CoChains(\CoisoModTriple_\coisoField{k})$.
    \end{definitionlist}
\end{definition}

Since morphisms of cochain complexes commute with the differential
$\delta$, it is easy to see that we obtain a new functor by
constructing the cohomology of the coisotropic complex.
\begin{proposition}[Coisotropic cohomology]
    \label{proposition:CoisoCohomology}%
    Let
    $\module{M}^\bullet \in \CoChains(\CoisoModTriple_\coisoField{k})$
    be a coisotropic cochain complex with differential $\delta$.  The
    maps
    \begin{align}
	\module{M}^i
        \longmapsto
        \functor{H}^i(\module{M}, \delta)
        =
        \ker \delta^i / \image \delta^{i-1}
    \end{align}
    for $i \in \Integers$ define a functor
    $\functor{H} \colon \CoChains(\CoisoModTriple_\coisoField{k})
    \longrightarrow \CoisoModTriple_\coisoField{k}^\bullet$.
\end{proposition}
\begin{remark}[(Regular) image]
    \label{remark:RegularImageAgain}%
    Note that the coisotropic cohomology is defined by using the
    \emph{image} of morphisms of coisotropic modules and not the
    \emph{regular image}.  However, choosing the regular image instead
    would not make a difference since the $\NULL$-component of the
    denominator is not used in the quotient of coisotropic modules,
    see \eqref{eq:QuotientsCoisotropicModules}.  Moreover, note that
    in general we can not decide whether $\ker \delta = \image \delta$
    by computing cohomology, but we can decide if
    $\ker \delta = \regimage \delta$ holds.
\end{remark}

Since graded coisotropic modules and coisotropic complexes are given
by $\Integers$-indexed families of coisotropic modules it should be
clear that applying the reduction functor in every degree yields
functors
$\red \colon \CoisoModTriple_\coisoField{k}^\bullet \to
\Modules_\field{k}^\bullet$ and
$\red \colon \CoChains(\CoisoModTriple_\coisoField{k}) \to
\CoChains(\Modules_\field{k})$.
It is now natural to investigate the relation between the cohomology
functor and the reduction functor. The following proposition shows
that reduction and cohomology functors commute.
\begin{proposition}[Cohomology commutes with reduction]
    \label{prop:CohomologyCommutesWithReduction}%
    There exists a natural isomorphism
    $\eta \colon \red \circ \functor{H} \Longrightarrow \functor{H}
    \circ \red$, i.e.
    \begin{equation}
	\begin{tikzcd}
            \CoChains(\CoisoModTriple_\coisoField{k})
            \arrow{r}{\functor{H}}
            \arrow{d}[swap]{\red}
            & \CoisoModTriple_\coisoField{k}
            \arrow{d}{\red}
            \arrow[Rightarrow]{dl}[swap]{\eta}\\
            \CoChains(\Modules_\field{k})
            \arrow{r}{\functor{H}}
            &\Modules_\field{k}
	\end{tikzcd}
    \end{equation}
    commutes.
\end{proposition}
\begin{proof}
    Define $\eta$ for every
    $\module{M} \in \CoChains(\CoisoModTriple_\coisoField{k})$
    by
    \begin{align*}
	\label{eq:prop:CohomologyCommutesWithReduction}
	\eta(\module{M}) \colon \functor{H}(\module{M})_\red
	\ni \big[[x]_\functor{H} \big]_\red
	\mapsto \big[ [x]_\red\big]_\functor{H}
	\in \functor{H}(\module{M}_\red).
    \end{align*}
    For $\delta_\Wobs^{i-1}y \in \image \delta_\Wobs^{i-1}$ we have
    $[\delta_\Wobs^{i-1}y ]_\red = \delta_\red^{i-1}[y]_\red$ and
    hence $[[\delta_\Wobs^{i-1}y]_\red]_\functor{H} = 0$.  Moreover,
    for $[x_0]_\functor{H} \in \functor{H}(\module{M})_\Null$
    we have $x_0 \in \module{M}^i_\Null$ and hence
    $[[x_0]_\red]_\functor{H} = 0$.  Thus $\eta$ is well-defined.
    Similarly, it can be shown that the inverse
    $\eta^{-1}(\module{M}) \colon
    \functor{H}(\module{M}\red) \longrightarrow
    \functor{H}(\module{M})_\red$ given by
    $[[x]_\red]_\functor{H} \mapsto [[x]_\functor{H}]_\red$ is
    well-defined.  Finally, for
    $\Phi \colon \module{M}^\bullet \longrightarrow
    \module{N}^\bullet$ we have
    \begin{align*}
	\Big(
        \eta(\module{N})
        \circ
        \big[[\Phi^i]_\functor{H}\big]_\red
        \Big)
	\big(\big[[x]_\functor{H}\big]_\red\big)
	&=
        \Big(\eta(\module{N}) \Big)
        \big(\big[[\Phi^i(x)]_\functor{H}\big]_\red\big)
        \\
	&=
        \big[[\Phi^i(x)]_\red\big]_\functor{H}
        \\
	&=
        \Big(
        \big[[\Phi^i]_\red\big]_\functor{H}
        \circ
        \eta(\module{M})
	\Big)
        \big(\big[[x]_\functor{H}\big]_\red\big),
    \end{align*}
    showing that $\eta$ is indeed a natural isomorphism.
\end{proof}
\begin{remark}
    \label{rem:Quis}%
    A morphism $\Phi \colon \module{M}^\bullet \to \module{N}^\bullet$
    of coisotropic cochain complexes is called a
    \emph{quasi-isomorphism} if the induced map $\functor{H}(\Phi)$ is
    an isomorphism of coisotropic modules. We remark that the
    reduction functor
    $\red \colon \CoChains(\CoisoModTriple_\coisoField{k}) \to
    \CoChains(\Modules_\field{k})$ maps quasi-isomorphisms of
    coisotropic cochain complexes to quasi-isomorphisms of cochain
    complexes.
\end{remark}

%
%

\section{Deformations via Coisotropic DGLAs}
\label{sec:CoisotropicDGLAs}

\subsection{Coisotropic DGLAs}

By a well-known principle of classical deformation theory, a
deformation problem is controlled by a certain differential graded Lie
algebra, see e.g. \cite{manetti:2009a}. Thus, the first step to
discuss the deformation theory of coisotropic algebras consists in
introducing a suitable notion of coisotropic differential graded Lie
algebra (DGLA) and a deformation functor in this realm.

\begin{definition}[Coisotropic differential graded Lie algebra]
    \label{definition:CoisoDGLA}%
    Let $\field{k}$ be a commutative unital ring.
    \begin{definitionlist}
    \item \label{item:CoisoDGLA} A \emph{coisotropic DGLA}
        $\liealg{g}$ over $\field{k}$ is a pair of DGLAs
        $(\liealg{g}_\Total^\bullet, [\argument,\argument]_\Total,
        \D_\Total)$ and
        $(\liealg{g}^\bullet_\Wobs,[\argument,\argument]_\Wobs,\D_\Wobs)$
        over $\field{k}$ together with a degree $0$ morphism
        $\iota_\liealg{g} \colon \liealg{g}_\Wobs^\bullet
        \longrightarrow \liealg{g}_\Total^\bullet$ of DGLAs and a
        graded Lie ideal
        $\liealg{g}_\Null^\bullet \subset \liealg{g}_\Wobs^\bullet$
        such that
        $\D_\Wobs (\liealg{g}_\Null^\bullet) \subseteq
        \liealg{g}_\Null^{\bullet +1}$.
    \item \label{item:CoisoDGLAMorph} For two coisotropic DGLAs
        $\liealg{g}$ and $\liealg{h}$, a \emph{morphism
          $\Phi \colon \liealg{g}^\bullet \longrightarrow
          \liealg{h}^\bullet$ of coisotropic DGLAs} is a pair of DGLA
        morphisms
        $\Phi_\Total \colon \liealg{g}^\bullet_\Total \to
        \liealg{h}^\bullet_\Total$ and
        $\Phi_\Wobs \colon \liealg{g}^\bullet_\Wobs \to
        \liealg{h}^\bullet_\Wobs$ such that
        $\Phi_\Total \circ \iota_\liealg{g} = \iota_\liealg{h} \circ
        \Phi_\Wobs$ and
        $\Phi_\Wobs(\liealg{g}^\bullet_\Null) \subseteq
        \liealg{h}^\bullet_\Null$.
    \item \label{item:CoisoDGLACat} The \emph{category of coisotropic
          DGLAs} will be denoted by $\CoisoDGLieAlgTriple$.
    \end{definitionlist}
\end{definition}

Note that a morphism of coisotropic DGLAs can equivalently be
understood as a morphism of coisotropic modules such that its
components are DGLA morphisms.  A coisotropic Lie algebra is a
coisotropic DGLA with trivial differential concentrated in degree $0$.
Similarly a coisotropic graded Lie algebra is a coisotropic DGLA with
trivial differential. Two important examples of coisotropic Lie
algebras are obtained as follows:
\begin{example}[Endomorphisms and derivations]
    \label{example:EndIsCoisoLie}%
    Let $\field{k}$ be a commutative unital ring.
    \begin{examplelist}
    \item \label{item:EndoAreLie} Let $\module{E}$ be a coisotropic
        $\field{k}$-module. The internal endomorphisms
        $\CoisoEnd_\coisoField{k}(\module{E})$ are a coisotropic Lie
        algebra given by the usual commutator
        $[\argument,\argument]^{\module{E}_\Total}$ on
        $\CoisoEnd_\coisoField{k}(\algebra{E})_\Total$ and the pair
        $([\argument, \argument]^{\module{E}_\Total}, [\argument,
        \argument]^{\module{E}_\Wobs})$ on
        $\CoisoEnd_\coisoField{k}(\algebra{E})_\Wobs$.
    \item \label{item:DerAreLie} Let $\algebra{A}$ be a coisotropic
        algebra over $\field{k}$. It is straightforward to see that
        $\CoisoDer(\algebra{A})$ is a coisotropic
        $\field{k}$-submodule of the coisotropic $\field{k}$-module
        $\CoisoEnd_\coisoField{k}(\algebra{A})$.  Moreover,
        $\CoisoDer(\algebra{A})$ is even a coisotropic Lie subalgebra
        of the coisotropic Lie algebra
        $\CoisoEnd_\coisoField{k}(\algebra{A})$. All canonical maps
        like \eqref{eq:ReducedDerivation} are in fact Lie morphisms.
    \end{examplelist}
\end{example}

Since every coisotropic DGLA $\liealg{g}$ is, in particular, a
coisotropic cochain complex we can always construct its corresponding
cohomology $\functor{H}(\liealg{g})$.  Moreover, every
morphism
$\Phi \colon \liealg{g}^\bullet \longrightarrow \liealg{h}^\bullet$ of
coisotropic DGLAs is a morphism of coisotropic cochain complexes and
therefore it induces a morphism
$\functor{H}(\Phi) \colon \functor{H}^\bullet(\liealg{g})
\longrightarrow \functor{H}^\bullet(\liealg{h})$ on cohomology.
Clearly, $\functor{H}(\liealg{g})$ is a coisotropic graded Lie
algebra and every induced morphism $\functor{H}(\Phi)$ is a morphism
of coisotropic graded Lie algebras.  If $\functor{H}(\Phi)$ is an
isomorphism we call $\Phi$ a \emph{coisotropic quasi-isomorphism}.
From \autoref{rem:Quis} it es clear that reduction of coisotropic
DGLAs preserves quasi-isomorphisms.

%
%

Following the standard way to define a deformation functor for a given
DGLA, we aim to define a Maurer-Cartan functor and to introduce a
notion of gauge equivalence. In order to define the Maurer-Cartan
elements in the coisotropic DGLA we first need an appropriate notion
of a \emph{coisotropic set}:
\begin{definition}[Coisotropic set]
    \label{definition:CoisoSet}%
    \
    \begin{definitionlist}
    \item \label{item:CoisoSet} A pair of sets $(M_\Total, M_\Wobs)$
        together with a map $\iota_M\colon M_\Wobs \to M_\Total$ and
        an equivalence relation $\sim$ on $M_\Wobs$ is called a
        \emph{coisotropic set}, denoted by
        $M = (M_\Total, M_\Wobs, \sim)$.
    \item \label{item:CoisoSetMorph} A \emph{morphism
          $f \colon M \to N$ of coisotropic sets} $M$ and $N$ is given
        by a pair of maps $f_\Total \colon M_\Total \to N_\Total$ and
        $f_\Wobs \colon M_\Wobs \to N_\Wobs$ such that
        $\iota_N \circ f_\Wobs = f_\Total \circ \iota_M$ and such that
        $f_\Wobs$ is compatible with the equivalence relations, i.e.
        $f(m) \sim_N f(m')$ for all $m,m' \in M$ with $m \sim_M m'$.
    \item \label{item:CoisoSetCat} The \emph{category of coisotropic
          sets} is denoted by $\CoisoSetTriple$.
    \end{definitionlist}
\end{definition}
\begin{remark}
    \label{remark:UnderlyingCoisoSets}%
    Every coisotropic $\field{k}$-module $\module{E}$, and hence every
    coisotropic algebra, coisotropic DGLA, etc., has an underlying
    coisotropic set in the sense that $\module{E}_\Wobs$ can be
    equipped with the equivalence relation induced by the submodule
    $\module{E}_\Null$.  In this sense coisotropic sets form the
    underlying structure for all the different notions of coisotropic
    algebraic structures.
\end{remark}

Given a coisotropic set we can clearly define a reduced one, as for
coisotropic modules, by taking the quotient
$M_\red = M / \mathord{\sim}$. This also yields a reduction functor
$\red \colon \CoisoSetTriple \to \Sets$.

We can now define the coisotropic set of Maurer-Cartan elements of a
coisotropic DGLA.  Recall that a Maurer-Cartan element in a DGLA
$\liealg{g}^\bullet$ is an element $\xi \in \liealg{g}^1$ satisfying the
Maurer-Cartan equation
\begin{equation}
	\D \xi + \frac{1}{2}[\xi, \xi] = 0.
\end{equation}
While up to here we did not have to make any further assumption about
the ring $\field{k}$ of scalars, from now on we assume
$\field{Q} \subseteq \field{k}$ in order to have a well-defined
Maurer-Cartan equation.
We denote by
$\MCset(\liealg{g})$ the set of all Maurer-Cartan elements of a DGLA.
\begin{definition}[Coisotropic set of Maurer-Cartan elements]
    \label{definition:CoisoMCElements}%
    Let $\liealg{g}$ be a coisotropic DGLA over a commutative
    unital ring $\field{k}$.  The \emph{coisotropic set
      $\MCset(\liealg{g})$ of Maurer-Cartan elements} of
    $\liealg{g}$ is given by
    \begin{equation}
        \label{eq:MCset}
	\MCset(\liealg{g})
        =
	\big(
        \MCset(\liealg{g}_\Total), \;
        \MCset(\liealg{g}_\Wobs), \;
        \sim_\MC
        \big),
    \end{equation}
    together with
    $\iota_\MC \colon \MCset(\liealg{g}_\Wobs) \longrightarrow
    \MCset(\liealg{g}_\Total)$ given by the map
    $\iota_\liealg{g} \colon \liealg{g}^\bullet_\Wobs \longrightarrow
    \liealg{g}^\bullet_\Total$ of $\liealg{g}$ and where the
    relation $\sim_\MC$ is defined by
    \begin{equation}
        \label{eq:MCEquivalence}
	\xi_1 \sim_\MC \xi_2 \iff \xi_1 - \xi_2 \in \liealg{g}_\Null^1
    \end{equation}
    for $\xi_1, \xi_2 \in \MCset(\liealg{g}_\Wobs)$.
\end{definition}
\begin{lemma}[Maurer-Cartan functor]
    \label{lem:MaurerCartanFunctor}%
    Mapping coisotropic DGLAs to their coisotropic sets of
    Maurer-Cartan elements defines a functor
    \begin{equation}
        \MCset\colon
        \CoisoDGLieAlgTriple \longrightarrow \CoisoSetTriple.
    \end{equation}
\end{lemma}
\begin{proof}
    Every morphism $\Phi \colon \liealg{g} \to \liealg{h}$ of
    coisotropic DGLAs induces maps
    $\Phi_\Total \colon \MCset(\liealg{g}_\Total) \to
    \MCset(\liealg{h}_\Total)$ and
    $\Phi_\Wobs \colon \MCset(\liealg{g}_\Wobs) \to
    \MCset(\liealg{h}_\Wobs)$.  Moreover, since
    $\Phi_\Wobs \colon \liealg{g}_\Wobs \to \liealg{h}_\Wobs$
    preserves the $\NULL$-component its induced map on
    $\MCset(\liealg{g}_\Wobs)$ maps equivalent elements to equivalent
    elements.
\end{proof}

As in the setting of DGLAs, for a given coisotropic DGLA
$(\liealg{g}, [\argument, \argument], \D)$ and a given Maurer-Cartan
element $\xi_0 \in \MCset(\liealg{g})_\Wobs$ we can always obtain a
twisted coisotropic DGLA by
$\liealg{g}_{\xi_0} = (\liealg{g}, [\argument,\argument], \D_{\xi_0})$
with
\begin{equation}
    \label{eq:TwistedDifferentialMC}
    \D_{\xi_0} := \D + [\xi_0, \argument].
\end{equation}
Here we are using the coevaluation morphism as mentioned in
\autoref{sec:CoisotropicModules}, \ref{InternalHom}.

Note that for any coisotropic DGLA $\liealg{g}$ and coisotropic
algebra $\algebra{A}$ the tensor product
$\liealg{g} \tensor \algebra{A}$ is again a coisotropic DGLA by the
usual construction.  For this observe that
$\liealg{g}_\Null \tensor \algebra{A}_\Wobs + \liealg{g}_\Wobs \tensor
\algebra{A}_\Null$ is indeed a Lie ideal in
$\liealg{g}_\Wobs \tensor \algebra{A}_\Wobs$.

Reformulating the equivalence of deformations of a given Maurer-Cartan
element in terms of its twisted coisotropic DGLA requires a notion of
a coisotropic gauge group.  For this reason we first introduce the
notion of a \emph{coisotropic group}:
\begin{definition}[Coisotropic group]
    \label{definition:CoisoGroup}%
    \begin{definitionlist}
    \item \label{item:CoisoGroup} A triple of groups
        $\group{G} = (\group{G}_\Total, \group{G}_\Wobs,
        \group{G}_\Null)$ together with a group homomorphism
        $\iota_\group{G} \colon \group{G}_\Wobs \to \group{G}_\Total$
        is called a \emph{coisotropic group} if
        $\group{G}_\Null \subseteq \group{G}_\Wobs$ is a normal
        subgroup.
    \item \label{item:CoisoGroupMorph} A \emph{morphism
          $\Phi \colon \group{G} \to \group{H}$ of coisotropic groups}
        $\group{G}$ and $\group{H}$ is given by a pair of group
        homomorphisms
        $\Phi_\Total \colon \group{G}_\Total \to \group{H}_\Total$ and
        $\Phi_\Wobs \colon \group{G}_\Wobs \to \group{H}_\Wobs$ such
        that
        $\iota_\group{H} \circ \Phi_\Wobs = \Phi_\Total \circ
        \iota_\group{G}$ and
        $\Phi_\Wobs(\group{G}_\Null) \subseteq \group{H}_\Null$.
    \item \label{item:CoisoGroupCat} The \emph{category of coisotropic
          groups} is denoted by $\CoisoGroupsTriple$.
    \end{definitionlist}
\end{definition}

Again, we obviously have a reduction functor
$\red \colon \CoisoGroupsTriple \to \Groups$ given by
$\group{G}_\red = \group{G}_\Wobs / \group{G}_\Null$.  Moreover, there
is a forgetful functor $\CoisoGroupsTriple \to \CoisoSetTriple$ by
only keeping the underlying sets and the equivalence relation induced
by the normal subgroup $\group{G}_\Null$.  It can be shown that the
automorphisms of a coisotropic set can be equipped with the structure
of a coisotropic group.  This leads to the definition of an
\emph{action} of a coisotropic group on a coisotropic set.
\begin{definition}[Action of coisotropic group]
    \label{definition:CoisoAction}%
    Let $\group{G}$ be a coisotropic group and $M$ a coisotropic set.
    An \emph{action} of $\group{G}$ on $M$ is given by an action
    $\Phi_\Total \colon \group{G}_\Total \times M_\Total \to M_\Total$
    of $\group{G}_\Total$ on $M_\Total$ and an action
    $\Phi_\Wobs \colon \group{G}_\Wobs \times M_\Wobs \to M_\Wobs$ of
    $\group{G}_\Wobs$ on $M_\Wobs$ such that
    $\iota_M \circ \Phi_\Wobs = \Phi_\Total \circ (\iota_\group{G}
    \times \iota_M)$ and $\Phi_g(m) \sim_M m$ for all
    $g \in \group{G}_\Null$ and $m \in M_\Wobs$.
\end{definition}
\begin{example}[Coisotropic groups and actions]
    \label{example:GroupsActions}
    \
    \begin{examplelist}
    \item \label{item:ExactSequence} Every short exact sequence of
        groups $1 \to \group{H} \to \group{G} \to \group{Q} \to 1$
        defines a coisotropic group
        $(\group{Q}, \group{G}, \group{H})$.
    \item \label{item:CoisoActionFromAction} Let
        $X = (X_\Total, X_\Wobs, \sim)$ a coisotropic set.  Let
        furthermore $\group{G}$ be a group acting on $X_\Total$ via
        $\Phi \colon \group{G} \times X_\Total \to X_\Total$.  Then
        $(\group{G}, \group{G}_{X_\Wobs}, \group{G}_\sim)$, with
        $\group{G}_{X_\Wobs}$ the stabilizer subgroup of the subset
        $X_\Wobs$ and $\group{G}_\sim$ the normal subgroup of
        $\group{G}_{X_\Wobs}$ consisting of all
        $g \in \group{X}_{X_\Wobs}$ such that $\Phi_g(p) \sim p$ for
        all $p \in X_\Wobs$, is a coisotropic group.  Clearly,
        $(\Phi,\Phi\at{\group{G}_{X_\Wobs}})$ gives a coisotropic
        action on $(X_\Total,X_\Wobs, \sim)$.
    \end{examplelist}
\end{example}

To define the coisotropic gauge group we either need to assume that
the DGLA we are starting with has additional properties, e.g. being
nilpotent, or we can use formal power series instead.  Since later on
we are interested in formal deformation theory, we will choose the
latter option.  For this let
$\coisoField{k}\formal{\lambda} = (\field{k}\formal{\lambda},
\field{k}\formal{\lambda},0)$ denote the coisotropic ring of formal
power series in $\field{k}$.

Then the formal power series $\module{E}\formal{\lambda}$ of any
coisotropic $\field{k}$-module $\module{E}$ form a coisotropic
$\field{k}\formal{\lambda}$-module as follows: we set
\begin{equation}
    \label{eq:EPowerSeries}
    \module{E}\formal{\lambda}
    =
    \left(
        \module{E}_\Total\formal{\lambda},
        \module{E}_\Wobs\formal{\lambda},
        \module{E}_\Null\formal{\lambda}
    \right),
\end{equation}
and use the canonical $\lambda$-linear extension
$\iota_{\module{E}\formal{\lambda}}$ of the previous map
$\iota_{\module{E}}\colon \module{E}_\Wobs \longrightarrow
\module{E}_\Total$. According to the usual convention, we denote this
extension simply by $\iota_{\module{E}}$. Note that in general
$\module{E}\formal{\lambda}$ is strictly larger than the tensor
product $\module{E} \tensor \field{k}\formal{\lambda}$: we still need
to take a $\lambda$-adic completion. This is the reasons that we
define $\module{E}\formal{\lambda}$ directly by
\eqref{eq:EPowerSeries}.

It is now easy to see that $\liealg{g}\formal{\lambda}$ is a
coisotropic DGLA for any coisotropic DGLA $\liealg{g}$ by
$\lambda$-linear extension of all structure maps. Similarly, we can
extend coisotropic algebras and their modules.

Note that the gauge action will require to have
$\field{Q} \subseteq \field{k}$ since we need (formal) exponential
series and the (formal) BCH series.
\begin{proposition}
    \label{prop:CoisoGaugeGroup}%
    Let
    $\liealg{g}$ be a coisotropic Lie algebra.  Then
    $\group{G}(\liealg{g})= (\lambda\liealg{g}_\Total\formal{\lambda},
    \lambda\liealg{g}_\Wobs\formal{\lambda},
    \lambda\liealg{g}_\Null\formal{\lambda})$ with multiplication
    $\bullet$ given by the Baker-Campbell-Hausdorff formula is a
    coisotropic group.
\end{proposition}
\begin{proof}
    The additional prefactor $\lambda$ makes all the BCH series
    $\lambda$-adically convergent.  The well-known group structures on
    $\liealg{g}_\Total\formal{\lambda}$ and
    $\liealg{g}_\Wobs\formal{\lambda}$ are given by the BCH formula
    and we clearly have a group morphism
    $\liealg{g}_\Wobs\formal{\lambda} \to
    \liealg{g}_\Total\formal{\lambda}$.  Finally, we need to show that
    $\lambda\liealg{g}_\Null\formal{\lambda}$ is a normal subgroup of
    $\lambda\liealg{g}_\Wobs\formal{\lambda}$.  For this let
    $\lambda g \in \lambda\liealg{g}_\Wobs\formal{\lambda}$ and
    $\lambda h \in \lambda\liealg{g}_\Null\formal{\lambda}$ be given.
    Since by the BCH formula
    $\lambda g \bullet \lambda h \bullet (\lambda g)^{-1} = \lambda
    g_0 + \lambda h_0 - \lambda g_0 + \lambda^2(\cdots)$, where all
    higher order terms are given by Lie brackets and
    $\liealg{g}_\Null$ is a Lie ideal in $\liealg{g}_\Wobs$, we see
    that
    $\lambda g \bullet \lambda h \bullet (\lambda g)^{-1} \in
    \lambda\liealg{g}_\Null\formal{\lambda}$.
\end{proof}

By abuse of notation we will write
$\group{G}(\liealg{g}) = \group{G}(\liealg{g}^0)$ for every
coisotropic DGLA $\liealg{g}$.  With the composition $\bullet$ on
$\group{G}(\liealg{g})$ defined by the Baker-Campbell-Hausdorff
formula it is immediately clear that every morphism
$\Phi \colon \liealg{g} \to \liealg{h}$ of coisotropic DGLAs induces a
morphism
$\group{G}(\Phi) \colon \group{G}(\liealg{g}) \to
\group{G}(\liealg{h})$ of the corresponding gauge groups, given by the
$\lambda$-linear extension of $\Phi$.  In other words, we obtain a
functor
$\group{G} \colon \CoisoDGLieAlgTriple \to \CoisoGroupsTriple$.

The usual gauge action of the formal group on the (formal)
Maurer-Cartan elements can now be extended to a coisotropic DGLA as
follows:
\begin{proposition}[Gauge action]
    \label{prop:GaugeAction}%
    Let
    $(\liealg{g}, [\argument,\argument], \D)$ be a coisotropic DGLA.
    Then the coisotropic group $\group{G}(\liealg{g})$ acts on the
    coisotropic set $\MCset(\lambda\liealg{g}\formal{\lambda})$ by
    \begin{equation}
        \label{eq:ActsTotal}
        \lambda g \acts_\Total \xi
        :=
        \E^{\lambda \ad_\Total(g)}(\xi)
        -
        \lambda \sum_{k=0}^{\infty}
        \frac{(\lambda \ad_\Total(g))^k}{(1+k)!}(\D_\Total g)
    \end{equation}
    for $\lambda g \in \group{G}(\liealg{g})_\Total$ and
    $\xi \in \MCset(\lambda\liealg{g}\formal{\lambda})_\Total$ as well
    as
    \begin{equation}
        \label{eq:ActsWobs}
	\lambda g \acts_\Wobs \xi
        :=
        \E^{\lambda \ad_\Wobs(g)}(\xi)
        -
        \lambda \sum_{k=0}^{\infty}
        \frac{(\lambda \ad_\Wobs(g))^k}{(1+k)!}(\D_\Wobs g)
    \end{equation}
    for $\lambda g \in \group{G}(\liealg{g})_\Wobs$ and
    $\xi \in \MCset(\lambda\liealg{g}\formal{\lambda})_\Wobs$.
\end{proposition}
\begin{proof}
    Clearly, $\acts_\Total$ and $\acts_\Wobs$ define actions of
    $\group{G}(\liealg{g})_\Total$ and $\group{G}(\liealg{g})_\Wobs$
    on $\MCset(\lambda\liealg{g}\formal{\lambda})_\Total$ and
    $\MCset(\lambda\liealg{g}\formal{\lambda})_\Wobs$, respectively,
    by classical results, see \cite{esposito:2015a}.  Moreover, writing
    out the exponential series and using the fact that
    $\ad(g) = [g, \argument]$ and $\D$ commute with $\iota_\liealg{g}$
    directly yields
    \begin{align*}
        \iota_\liealg{g}
        \left(\lambda g \acts_\Wobs \xi\right)
        &=
        \E^{\lambda \ad_\Total(\iota_\liealg{g}(g))}(\iota_\liealg{g}(\xi))
        -
        \lambda \sum_{k=0}^{\infty}
        \frac{(\lambda \ad_\Total(\iota_\liealg{g}(g)))^k}{(1+k)!}
        (\D_\Total \iota_\liealg{g}(g))
        \\
        &=
        \lambda \iota_\liealg{g}(g) \acts_\Total \iota_\liealg{g}(\xi).
    \end{align*}
    Finally, we have for any
    $\lambda g \in \group{G}(\liealg{g})_\Null$ and
    $\xi \in \MCset(\lambda\liealg{g}\formal{\lambda})_\Wobs$
    \begin{align*}
	\E^{\lambda \ad_\Wobs(g)}(\xi) - \xi
	&=
        \sum_{k=0}^{\infty} \frac{\lambda^k}{k!} (\ad_\Wobs(g))^k(\xi)
        -
        \lambda \sum_{k=0}^{\infty}
        \frac{(\lambda \ad_\Wobs(g))^k}{(1+k)!}(\D_\Wobs g)
        - \xi
        \\
        &=
        \sum_{k=1}^{\infty} \frac{\lambda^k}{k!} (\ad_\Wobs(g))^k (\xi)
        -
        \lambda \sum_{k=0}^{\infty}
        \frac{(\lambda \ad_\Wobs(g))^k}{(1+k)!}(\D_\Wobs g)
	\in \lambda\liealg{g}_\Null\formal{\lambda},
	\end{align*}
        since $\D_\Wobs g \in \liealg{g}_\Null\formal{\lambda}$ and
        $\ad_\Wobs(g)(\xi) \in \liealg{g}_\Null\formal{\lambda}$.
\end{proof}

\subsection{Deformation functor and reduction}

Maurer-Cartan elements are said to be \emph{equivalent} if they lie in
the same orbit of the gauge action.  Hence the object of interest for
deformation theory is not the set of Maurer-Cartan elements itself but
its set of equivalence classes.  More precisely let us denote by
$\Def(\liealg{g})$ the pair given by
\begin{align}
    \Def(\liealg{g})_\Total
    &:=
    \MCset(\lambda\liealg{g}\formal{\lambda})_\Total
    /
    \group{G}(\liealg{g})_\Total
    \\
    \shortintertext{and}
    \Def(\liealg{g})_\Wobs
    &:=
    \MCset(\lambda\liealg{g}\formal{\lambda})_\Wobs
    /
    \group{G}(\liealg{g})_\Wobs
\end{align}
with an equivalence relation on $\Def(\liealg{g})_\Wobs$ defined by
\begin{equation}
    \label{eq:GaugeEquivalence}
    [\xi_1] \sim [\xi_2]
    :\Longleftrightarrow
    \xi_1 \sim_\MC \xi_2.
\end{equation}
\begin{proposition}
    \label{proposition:DefCoisoSet}%
    Let $\liealg{g}$ be a coisotropic
    DGLA.  Then $\Def(\liealg{g})$ is a coisotropic set.
\end{proposition}
\begin{proof}
    By \autoref{prop:GaugeAction} we know that the action of
    $\group{G}(\liealg{g})$ is compatible with
    $\iota_\MC \colon \MCset(\lambda\liealg{g}\formal{\lambda})_\Wobs
    \to \MCset(\lambda\liealg{g}\formal{\lambda})_\Total$, hence
    $\iota_\MC$ descends to a morphism
    $\iota_\MC \colon \Def(\liealg{g})_\Wobs \to
    \Def(\liealg{g})_\Total$.  To see that \eqref{eq:GaugeEquivalence}
    yields a well-defined equivalence relation suppose that
    $\lambda g \acts \xi_1$ is another representative of $[\xi_1]$.
    Then again by \autoref{prop:GaugeAction} we know that
    $\lambda g \acts \xi_1 - \xi_1 \in
    \lambda\liealg{g}\formal{\lambda}_\Null$ and thus
    $\lambda g \acts \xi_1 \sim_\MC \xi_1$, showing that
    \eqref{eq:GaugeEquivalence} is well-defined.
\end{proof}

We have seen in \autoref{lem:MaurerCartanFunctor} that morphisms of
coisotropic DGLAs induce morphisms between the corresponding
coisotropic sets of Maurer-Cartan elements.  This is still true after
taking the quotient by the coisotropic gauge group.
\begin{proposition}
    \label{proposition:DefFunctor}%
    Mapping coisotropic DGLAs $\liealg{g}$ to the quotient set
    $\Def(\liealg{g})$ defines a functor
    \begin{equation}
        \label{eq:DefFunctor}
        \Def \colon \CoisoDGLieAlgTriple \to \CoisoSetTriple.
    \end{equation}
\end{proposition}
\begin{proof}
    Given a morphism $\Phi \colon \liealg{g} \to \liealg{h}$ of
    coisotropic DGLAs we get morphisms
    $\MCset(\Phi) \colon \MCset(\lambda \liealg{g}\formal{\lambda})
    \to \MCset(\lambda\liealg{h}\formal{\lambda})$ and
    $\group{G}(\Phi) \colon \group{G}(\liealg{g}) \to
    \group{G}(\liealg{h})$ as shown in
    \autoref{lem:MaurerCartanFunctor} and after
    \autoref{prop:CoisoGaugeGroup}.  Then we have
    \begin{align*}
        \MCset(\Phi) ( \lambda g \acts_\Wobs \xi)
        &=
        \Phi_\Wobs \left(
            \E^{\lambda\ad_\Wobs(g)}(\xi)
            -
            \lambda \sum_{k=0}^{\infty}
            \frac{(\lambda \ad_\Wobs(g))^k}{(1+k)!}(\D_\Wobs g)
        \right)
        \\
        &=
        \E^{\lambda(\ad_\Wobs(\Phi_\Wobs(g)))}(\Phi_\Wobs(\xi))
        -
        \lambda \sum_{k=0}^{\infty}
        \frac{(\lambda \ad_\Wobs(\Phi_\Wobs(g)))^k}{(1+k)!}
        (\D_\Wobs \Phi_\Wobs(g))
        \\
        &=
        \group{G}(\Phi)(\lambda g) \acts \MCset(\Phi)(\xi),
    \end{align*}
    and similar for the $\TOTAL$-component, showing that
    $\MCset(\Phi)$ is equivariant along $\group{G}(\Phi)$ and hence
    inducing a morphism $\Def(\Phi)$ as needed.
\end{proof}

The question arises if the above constructions of the coisotropic set
of Maurer-Cartan elements, the coisotropic gauge group and the
deformation functor commute with reduction.  The next theorem shows
that this is partially true, in the sense that at least an injective
natural transformation exists.
\begin{theorem}[Gauge group and reduction]
    \label{thm:defvsred}%
    Let $\field{k}$ be a commutative ring containing $\field{Q}$.
    \begin{theoremlist}
    \item \label{item:MCredNatural} There exists an injective natural
        transformation
        $\eta \colon \red \circ \MCset \Longrightarrow \MCset \circ
        \red$, i.e.
        \begin{equation}
            \begin{tikzcd}
                \CoisoDGLieAlgTriple
                \arrow{r}{\MCset}
                \arrow{d}[swap]{\red}
                & \CoisoSetTriple
                \arrow{d}{\red}
                \arrow[Rightarrow]{dl}[swap]{\eta}\\
                \DGLieAlg
                \arrow{r}{\MCset}
                &\Sets
            \end{tikzcd}
	\end{equation}
        commutes with $\eta$ injective.
    \item \label{item:GredNatural} There exists a natural isomorphism
        $\eta \colon \red \circ \group{G} \Longrightarrow \group{G}
        \circ \red$, i.e.
        \begin{equation}
            \begin{tikzcd}
                \CoisoDGLieAlgTriple
                \arrow{r}{\group{G}}
                \arrow{d}[swap]{\red}
                & \CoisoGroupsTriple
                \arrow{d}{\red}
                \arrow[Rightarrow]{dl}[swap]{\eta}\\
                \DGLieAlg
                \arrow{r}{\group{G}}
                &\Groups
            \end{tikzcd}
	\end{equation}
        commutes with $\eta$ injective.
    \item \label{item:DefredNatural} There exists an injective natural
        transformation
        $\eta \colon \red \circ \Def \Longrightarrow \Def \circ \red$,
        i.e.
        \begin{equation}
            \begin{tikzcd}
		\CoisoDGLieAlgTriple
		\arrow{r}{\Def}
		\arrow{d}[swap]{\red}
		& \CoisoSetTriple
		\arrow{d}{\red}
		\arrow[Rightarrow]{dl}[swap]{\eta}\\
		\DGLieAlg
		\arrow{r}{\Def}
		&\Sets
            \end{tikzcd}
	\end{equation}
        commutes with $\eta$ injective.
    \end{theoremlist}
\end{theorem}
\begin{proof}
    \begin{theoremlist}
    \item In the following we denote by $[\argument]_\MC$ the
        equivalence classes of elements in $\MCset(\liealg{g}_\Wobs)$
        and by $[\argument]_\liealg{g}$ the equivalence classes of
        elements in $\liealg{g}_\Wobs$.  For any coisotropic DGLA
        $\liealg{g}$ define
        $\eta_\liealg{g} \colon \MCset(\liealg{g})_\red \to
        \MCset(\liealg{g}_\red)$ by
        $\eta_\liealg{g}([\xi]_\MC) = [\xi]_\liealg{g}$.  This map is
        well-defined since $[\xi]_\MC \subseteq [\xi]_\liealg{g}$ and
        \begin{equation*}
            \D_\red [\xi]_\liealg{g}
            +
            \big[ [\xi]_\liealg{g}, [\xi]_\liealg{g} \big]_\red
            =
            \big[ \D_\Wobs \xi + [\xi,\xi]_\Wobs \big]_\liealg{g}
            =
            [0]_\liealg{g}
	\end{equation*}
        for every $\xi \in \MCset(\liealg{g}_\Wobs)$.  To show that
        $\eta_\liealg{g}$ is injective let
        $[\xi_1]_\MC, [\xi_2]_\MC \in \MCset(\liealg{g})_\red$ be
        given such that $[\xi_1]_\liealg{g} = [\xi_2]_\liealg{g}$.
        Then $\xi_2 \in [\xi_1]_\liealg{g}$ and hence
        $\xi_1-\xi_2 \in \liealg{g}_\Null^1$.  Thus by definition
        $\xi_1 \sim_\MC \xi_2$ and therefore
        $[\xi_1]_\MC = [\xi_2]_\MC$.  To show naturality of $\eta$ let
        a morphism $\Phi \colon \liealg{g} \to \liealg{h}$ of
        coisotropic DGLAs be given.  This induces morphisms
        $\Phi \colon \MCset(\liealg{g})_\red \to
        \MCset(\liealg{h})_\red$ and
        $\Phi \colon \MCset(\liealg{g}_\red) \to
        \MCset(\liealg{h}_\red)$ by applying $\Phi_\Wobs$ to
        representatives.  Then we have
        \begin{equation*}
            (\eta_\liealg{h} \circ \Phi)([\xi]_\MC)
            = \eta_\liealg{h}([\Phi_\Wobs(\xi)]_\MC)
            = [\Phi_\Wobs(\xi)]_\liealg{h}
            = \Phi([\xi]_\liealg{g})
            = \Phi(\eta_\liealg{g}([\xi]_\MC)),
        \end{equation*}
        showing that $\eta$ is natural.
    \item Then
        $\eta_\liealg{g} \colon \group{G}(\liealg{g})_\red \to
        \group{G}(\liealg{g}_\red)$ given by
        $[\lambda g]_{\group{G}}\mapsto \lambda [g]_\liealg{g}$, where
        $[g]_\liealg{g}$ denotes the equivalence class of $g$ in
        $\liealg{g}_\red$, is well-defined.  Indeed, $\eta_\liealg{g}$
        is just the $\lambda$-linear extension of the obvious identity
        $\liealg{g}_\Wobs / \liealg{g}_\Null = \liealg{g}_\red$.
        Moreover, $\eta_\liealg{g}$ is a group morphism, since
        $[\argument]_\liealg{g} \colon \liealg{g}_\Wobs \to
        \liealg{g}_\red$ is a morphism of DGLAs and $\bullet$ is given
        by sums of iterated brackets.
        Naturality follows directly.
    \item Let $\liealg{g} \in \CoisoDGLieAlgTriple$ be a coisotropic
        DGLA.  Define
        $\eta_\liealg{g} \colon \Def(\liealg{g})_\red \to
        \Def(\liealg{g}_\red)$ by
        $\left[ [\lambda g]_{\group{G}} \right]_\Def \mapsto \left[
            [\lambda g]_\MC \right]_{\group{G}}$ where
        $[\argument]_{\group{G}}$ denotes the equivalence class
        induced by the action of the gauge group, $[\argument]_\Def$
        denotes the equivalence class given by the equivalence
        relation on $\Def(\liealg{g})_\Wobs$ and $[\argument]_\MC$
        denotes the equivalence class given by the equivalence
        relation on $\MCset(\lambda\liealg{g}_\red\formal{\lambda})$.
        Now suppose that $[\lambda g']_{\group{G}}$ is another
        representative of $[[\lambda g]_{\group{G}}]_\Def$, hence
        $\lambda g \sim_\MC \lambda g'$ by \eqref{eq:GaugeEquivalence}
        and thus $\eta_\liealg{g}$ is well-defined.  For the
        injectivity let
        $[[\lambda g ]_\MC]_{\group{G}} = [[\lambda
        g']_\MC]_{\group{G}}$ be given.  Then
        $[\lambda g]_\MC = \lambda [h] \acts [\lambda g']_\MC =
        [\lambda h \acts_\Wobs \lambda g' ]_\MC$ for some
        $[h] \in \liealg{g}_\red\formal{\lambda}$ and therefore
        $[[\lambda g]_{\group{G}}]_\Def = [[\lambda
        g']_{\group{G}}]_\Def$.
        Naturality follows as above.
    \end{theoremlist}
\end{proof}

%
%

\section{Deformation Theory of Coisotropic Algebras}
\label{sec:DeformationTheoryOfCoisotropicAlgebras}

%
%

\subsection{Deformations of Coisotropic Algebras}
\label{subsec:DefCoisoAlgs}

In formal deformation quantization one is interested in algebras of
formal power series over a ring $\field{k}[[\lambda]]$,
e.g. $(\Cinfty(M)[[\lambda]], \star)$ as algebra over
$\mathbb{C}[[\lambda]]$ for a Poisson manifold $M$ with star product
$\star$. For this reason, we consider deformations of a coisotropic
algebra $\algebra{A}$ with respect to the augmented (coisotropic) ring
$\coisoField{k}\formal{\lambda} =
(\field{k}\formal{\lambda},\field{k}\formal{\lambda},0)$.  Given a
coisotropic $\field{k}$-algebra
$\algebra{A} \in \CoisoAlgTriple_\coisoField{k}$, we can define a
formal deformation to be a coisotropic
$\field{k}\formal{\lambda}$-algebra $\algebra{B}$ together with an
isomorphism
$\alpha \colon \cl(\algebra{B}) \longrightarrow \algebra{A}$.  Here
$\cl(\algebra{B})$ denotes the classical limit as introduced in
\cite{dippell.esposito.waldmann:2019a}: The classical limit of a
coisotropic $\field{k}\formal{\lambda}$-algebra $\algebra{B}$ is the
coisotropic $\field{k}$-algebra defined as
$\cl(\algebra{B}) =\algebra{B}/\lambda\algebra{B} =
(\algebra{B}_\Total/\lambda\algebra{B}_\Total, \algebra{B}_\Wobs /
\lambda\algebra{B}_\Wobs, \algebra{B}_\Null /
\lambda\algebra{B}_\Wobs)$.  It is easy to see that this definition
agrees with the one from deformation via Artin rings, see
e.g. \cite{manetti:2009a}.  Usually, one is interested in more
specific deformations, namely those that are e.g. free
$\field{k}$-modules.  This leads us to the following definition:
\begin{definition}[Deformation of coisotropic algebra]
    \label{def:DeformationOfCoisoAlgebra}%
    Let $\algebra{A} \in \CoisoAlgTriple_\coisoField{k}$ be a
    coisotropic algebra.  A \emph{(formal associative) deformation} of
    $\algebra{A}$ is given by an associative multiplication
    $\mu \colon \algebra{A}\formal{\lambda} \tensor
    \algebra{A}\formal{\lambda} \longrightarrow
    \algebra{A}\formal{\lambda}$ on $\algebra{A}\formal{\lambda}$
    turning it into a coisotropic $\field{k}\formal{\lambda}$-algebra
    such that
    $\cl(\algebra{A}\formal{\lambda},\mu) \simeq \algebra{A}$.
\end{definition}

Let us comment on this definition. First recall that we have
\begin{equation}
    \label{eq:APowerSeries}
    \algebra{A}\formal{\lambda}
    =
    \left(
        \algebra{A}_\Total\formal{\lambda},
        \algebra{A}_\Wobs\formal{\lambda},
        \algebra{A}_\Null\formal{\lambda}
    \right)
\end{equation}
with the structure map
$\iota_{\algebra{A}\formal{\lambda}} = \iota_{\algebra{A}}$ being just
the $\lambda$-linear extension of the previous map according to
\eqref{eq:EPowerSeries}. Then we have two formal associative
deformations $\mu_\Total$ and $\mu_\Wobs$ for
$\algebra{A}_\Total\formal{\lambda}$ and
$\algebra{A}_\Wobs\formal{\lambda}$ of the form
$\mu_\Total = (\mu_\Total)_0 + \lambda (\mu_\Total)_1 +
\lambda^2(\dotsc)$ and
$\mu_\Wobs = (\mu_\Wobs)_0 + \lambda (\mu_\Wobs)_1 +
\lambda^2(\dotsc)$, respectively, such that the \emph{undeformed} map
$\iota_{\algebra{A}}$ is an algebra homomorphism and such that
$\algebra{A}_\Null\formal{\lambda}$ is a two-sided ideal in
$\algebra{A}_\Wobs\formal{\lambda}$ with respect to $\mu_\Wobs$.  Note
that we insist on the $\algebra{A}_\Wobs$ and $\algebra{A}_\Null$
being the \emph{same} up to taking formal series. Also the algebra
morphism $\iota_{\algebra{A}}$ is \emph{not} deformed.

One particular scenario we will be interested in the context of
deformation quantization of phase space reduction is the following:
\begin{example}
    \label{example:DeformationCTriple}%
    For convenience, we will assume that $\field{k}$ is actually a
    field and not just a ring.  Let
    $\algebra{A} = (\algebra{A}_\Total \supseteq \algebra{A}_\Wobs
    \supseteq \algebra{A}_\Null)$ be a coisotropic triple such that
    $\algebra{A}_\Null \subseteq \algebra{A}_\Total$ is a left ideal
    and $\algebra{A}_\Wobs \subseteq \normalizer(\algebra{A}_\Null)$
    is a unital subalgebra of the normalizer of this left ideal. In
    particular, $\iota_{\algebra{A}}$ is just the inclusion. Consider
    now a formal associative deformation $\mu_\Total$ of
    $\algebra{A}_\Total$ with the additional property that the formal
    series $\algebra{A}_\Null\formal{\lambda}$ are still a left ideal
    inside $\algebra{A}_\Total\formal{\lambda}$ with respect to
    $\mu_\Total$. Moreover, assume that the normalizer
    $\qalgebra{A}_\Wobs =
    \normalizer_{\mu_\Total}(\algebra{J}\formal{\lambda}) \subseteq
    \algebra{A}_\Total\formal{\lambda}$ with respect to $\mu_\Total$
    satisfies
    \begin{equation}
        \label{eq:ConditionqWobs}
        \cl(\qalgebra{A}_\Wobs) \subseteq \algebra{A}_\Wobs.
    \end{equation}
    This would be automatically true if $\algebra{A}_\Wobs$ coincides
    with the undeformed normalizer but poses an additional condition
    otherwise.

    It is now easy to check that
    $\qalgebra{A}_\Wobs \subseteq \algebra{A}_\Total\formal{\lambda}$
    is a \emph{closed} subspace with respect to the $\lambda$-adic
    topology. Moreover, if $\lambda a \in \qalgebra{A}_\Wobs$ for some
    $a \in \algebra{A}_\Total[[\lambda]]$ we can conclude
    $a \in \qalgebra{A}\formal{\lambda}$. Hence
    $\qalgebra{A}_\Wobs \subseteq \algebra{A}\formal{\lambda}$ is a
    deformation of a subspace in the sense of
    \cite[Def.~30]{bordemann.herbig.waldmann:2000a}, i.e. we have a
    subspace $\algebra{D} \subseteq \algebra{A}_\Total$ and linear
    maps $q_r\colon \algebra{D} \longrightarrow \algebra{A}_\Total$
    for $r \in \mathbb{N}$ such that
    \begin{equation}
        \label{eq:qWobsAreDeformedSubspace}
        \qalgebra{A}_\Wobs = q (\algebra{D}\formal{\lambda})
    \end{equation}
    where $q = \iota_{\algebra{D}} + \sum_{r=1}^\infty \lambda^r q_r$
    with $\iota_{\algebra{D}}$ being the canonical inclusion of the
    subspace. By our assumption,
    $\algebra{D} \subseteq \algebra{A}_\Wobs$ but the inclusion could
    be proper. Moreover, since by our assumption
    $\algebra{A}_\Null\formal{\lambda} \subseteq
    \normalizer(\algebra{A}_\Null\formal{\lambda}) =
    \qalgebra{A}_\Wobs$, we have
    $\algebra{A}_\Null \subseteq \algebra{D}$.

    Since we work over a field, we can find a complement
    $\algebra{C} \subseteq \algebra{D}$ such that
    $\algebra{A}_\Null \oplus \algebra{C} = \algebra{D}$. This allows
    to redefine the maps $q_r$ to
    \begin{equation}
        \label{eq:qprimer}
        q'_r\at{\algebra{C}} = q_r\at{\algebra{C}}
        \quad
        \textrm{and}
        \quad
        q'_r\at{\algebra{A}_\Null} = 0.
    \end{equation}
    The resulting map $q'$ then satisfies
    $q'(\algebra{D}\formal{\lambda}) = \qalgebra{A}_\Wobs$ and
    $q'\at{\algebra{A}_\Null} = \id_{\algebra{A}_\Null}$. We can then
    use $q'$ to pass to a new deformation $\mu'_\Total$ of
    $\algebra{A}_\Total$ with the property that
    $\algebra{A}_\Null\formal{\lambda}$ is still a left ideal in
    $\algebra{A}_\Total\formal{\lambda}$ with respect to
    $\mu'_\Total$ and the normalizer of this left ideal is now given
    by
    $\algebra{D}\formal{\lambda} \subseteq
    \algebra{A}_\Total\formal{\lambda}$. It follows that
    $\mu'_\Total$ provides a deformation of the coisotropic triple
    $(\algebra{A}_\Total, \algebra{D}, \algebra{A}_\Null)$ in the
    sense of Definition~\ref{def:DeformationOfCoisoAlgebra}.

    Of course, it might happen that
    $\algebra{D} \ne \algebra{A}_\Wobs$ and hence this construction
    will not provide a deformation of the original coisotropic triple,
    in general. It turns out that this can be controlled as follows:
    we assume in addition that the deformed normalizer
    $\qalgebra{A}_\Wobs$ is \emph{large enough} in the sense that the
    classical limit
    \begin{equation}
        \label{eq:LargeEnoughNormalizer}
        \cl\colon \qalgebra{A}_\red
        =
        \qalgebra{A}_\Wobs \big/ (\algebra{A}_\Null\formal{\lambda})
        \longrightarrow
        \algebra{A}_\red = \algebra{A}_\Wobs \big/ \algebra{A}_\Null
    \end{equation}
    between the reduced algebras is \emph{surjective}. As $\field{k}$
    is a field, this gives us a split
    $Q\colon \algebra{A}_\red \longrightarrow \qalgebra{A}_\red$ which
    we can extend $\lambda$-linearly to
    \begin{equation}
        \label{eq:QuantizedTheReduction}
        Q\colon
        \algebra{A}_\red\formal{\lambda}
        \longrightarrow
        \qalgebra{A}_\red.
    \end{equation}
    It is then easy to see that this is in fact a
    $\field{k}\formal{\lambda}$-linear isomorphism. It follows, that
    in this case we necessarily have
    \begin{equation}
        \label{eq:DisAWobs}
        \algebra{D} = \algebra{A}_\Wobs.
    \end{equation}
    Thus the previous construction gives indeed a deformation
    $\mu'_\Total$ of the original coisotropic triple. This seemingly
    very special situation will turn out to be responsible for one of
    the main examples from deformation quantization.
\end{example}

%
%

\subsection{Coisotropic Hochschild Cohomology}
\label{sec:CoisotropicHochschildCohomology}

From now on we assume that $\field{Q} \subseteq \field{k}$.
Let $\module{M}, \module{N} \in \CoisoModTriple_\coisoField{k}$ be
coisotropic modules.  We define for any $n \in \Naturals$
\begin{equation}
    \label{eq:CnDef}
    \mathrm{C}^n(\module{M}, \module{N})
    =
    \CoisoHom_\coisoField{k}(\module{M}^{\tensor n},\module{N})
\end{equation}
with $\CoisoHom_\coisoField{k}$ denoting the internal Hom as usual.
Recall that
\begin{align*}
    \mathrm{C}^n(\module{M},\module{N})_\Total
    &=
    \Hom_\field{k}(\module{M}_\Total^{\tensor n}, \module{N}_\Total),
    \\
    \mathrm{C}^n(\module{M},\module{N})_\Wobs
    &=
    \Hom_\coisoField{k}(\module{M}^{\tensor n},\module{N}),
    \\
    \mathrm{C}^n(\module{M},\module{N})_\Null
    &=
    \big\{
    (f_\Total,f_\Wobs)
    \in
    \mathrm{C}^n(\module{M}^{\tensor n},\module{N})
    \; \big| \;
    f_\Wobs(\module{M}^{\tensor n}_\Wobs) \subseteq \module{N}_\Null
    \big\}
\end{align*}
with
$\iota_n \colon \mathrm{C}^n(\module{M},\module{N})_\Wobs \ni
(f_\Total,f_\Wobs) \mapsto f_\Total \in
\mathrm{C}^n(\module{M},\module{N})_\Total$.  Note that a morphism
$f = (f_\Total, f_\Wobs) \in
\mathrm{C}^n(\module{M},\module{N})_\Wobs$ fulfils
$f_\Wobs((\module{M}^{\tensor n})_\Null) \subseteq \module{N}_\Null$
where, by definition of the tensor product, we have
\begin{equation}
    \label{eq:NullOfTensorPower}
    (\module{M}^{\tensor n})_\Null
    =
    \sum_{i=1}^{n}
    \module{M}^{\tensor i-1}_\Wobs
    \tensor
    \module{M}_\Null
    \tensor
    \module{M}_\Wobs^{\tensor n-i}.
\end{equation}
In other words, $f_\Wobs$ maps to $\module{N}_\Null$ if at least one
tensor factor comes from $\module{M}_\Null$.  This clearly defines a
graded coisotropic $\field{k}$-module
$\mathrm{C}^\bullet(\module{M},\module{N})$.

Let us now consider the case $\module{N} = \module{M}$.  Then we write
$\mathrm{C}^\bullet(\module{M}) =
\mathrm{C}^\bullet(\module{M},\module{M})$.  We now want to transfer
the Gerstenhaber algebra structure of the classical Hochschild complex
to $\mathrm{C}^\bullet(\module{M})$.  For this denote by
$[\argument,\argument]^{\module{M}_\Total}$ and
$[\argument,\argument]^{\module{M}_\Wobs}$ the Gerstenhaber brackets
for the modules $\module{M}_\Total$ and $\module{M}_\Wobs$,
respectively.  Then we need to show that
$[\argument,\argument]^{\module{M}_\Wobs}$ preserves the
$\NULL$-components.  This follows directly from the usual formula for
the Gerstenhaber bracket, see \cite{gerstenhaber:1963a}.
\begin{definition}[Gerstenhaber bracket]
    \label{definition:CoisoGerstenhaberBracket}%
    Let $\module{M} \in \CoisoModTriple_\coisoField{k}$.  Then the
    morphism
    $[\argument, \argument] \colon \mathrm{C}^\bullet(\module{M})
    \tensor \mathrm{C}^\bullet(\module{M}) \to
    \mathrm{C}^\bullet(\module{M})$ of coisotropic $\field{k}$-modules
    defined by
    $[\argument, \argument] =
    ([\argument,\argument]^{\module{M}_\Total},
    ([\argument,\argument]^{\module{M}_\Total},
    [\argument,\argument]^{\module{M}_\Wobs}))$ is called the
    \emph{coisotropic Gerstenhaber bracket}.
\end{definition}
Since $[\argument, \argument]^{\module{M}_\Total}$ and
$[\argument, \argument]^{\module{M}_\Wobs}$ induce graded Lie algebra
structures on the classical Hochschild complexes of
$\module{M}_\Total$ and $\module{M}_\Wobs$ it is easy to see that
$\mathrm{C}^\bullet(\module{M})$ together with the coisotropic
Gerstenhaber bracket $[\argument, \argument]$ forms a graded
coisotropic Lie algebra.

\begin{remark}
    \label{remark:GerstenhaberFromPreLie}%
    The coisotropic Gerstenhaber bracket can also be derived from a
    coisotropic pre-Lie algebra structure on
    $\mathrm{C}^\bullet(\module{M})$, which in turn results from a
    sort of partial composition.  These partial compositions can be
    interpreted as the usual endomorphism operad structure of
    $\module{M}$ in $\CoisoModTriple_\coisoField{k}$.  The theory of
    operads in the (non-abelian) category
    $\CoisoModTriple_\coisoField{k}$ will be the subject of a future
    project.
\end{remark}

As in the standard theory of deformation of associative algebras, we
can characterize associative multiplications by using the Gerstenhaber
bracket.
\begin{lemma}
    \label{lem:AssoCoisotropicMultiplication}%
    Let $\module{M} \in \CoisoModTriple_\coisoField{k}$ be a
    coisotropic module.  Then a morphism
    $\mu \colon \module{M} \tensor \module{M} \longrightarrow
    \module{M}$ of coisotropic $\field{k}$-modules is an associative
    coisotropic algebra structure on $\module{M}$ if and only if
    \begin{equation}
	\label{eq:AssociativityViaMC}
	[\mu, \mu]_\Wobs = 0.
    \end{equation}
\end{lemma}
\begin{proof}
    First, note that a coisotropic morphism
    $\mu\colon \module{M} \tensor \module{M} \to \module{M}$ is an
    element in $\mathrm{C}^2(\module{M})_\Wobs$ and hence consists of
    a pair $(\mu_\Total, \mu_\Wobs)$ and
    $[\argument, \argument]_\Wobs = ([\argument,
    \argument]^{\module{M}_\Total},
    [\argument,\argument]^{\module{M}_\Wobs})$.  From the classical
    theory for associative algebras we know that $\mu_\Total$ and
    $\mu_\Wobs$ are associative multiplications if and only if
    $[\mu_\Total,\mu_\Total]^{\module{M}_\Total} = 0$ and
    $[\mu_\Wobs, \mu_\Wobs]^{\module{M}_\Wobs} = 0$ holds.
\end{proof}

Note that \eqref{eq:AssociativityViaMC} only involves the
$\WOBS$-component of the coisotropic Gerstenhaber bracket
$[\argument, \argument]$.  Using the coisotropic structure of
$\mathrm{C}^2(\module{M})$ we get
$\iota_2(\mu) = \mu_\Total \in \mathrm{C}^2(\module{M})_\Total$, from
which directly $[\mu_\Total,\mu_\Total]_\Total = 0$ follows.

Let us now move from a module $\module{M}$ to an algebra
$(\algebra{A},\mu)$.  Then we can use the multiplication to construct
a differential on $\mathrm{C}^\bullet(\algebra{A})$.
\begin{proposition}[Coisotropic Hochschild differential]
    \label{proposition:CoisoHochschildDifferential}%
    Let $(\algebra{A}, \mu) \in \CoisoAlgTriple_\coisoField{k}$ be a
    coisotropic algebra.  Then the morphism
    $\delta\colon \mathrm{C}^\bullet(\algebra{A}) \to
    \mathrm{C}^{\bullet +1}(\algebra{A})$ of coisotropic
    $\field{k}$-modules, defined by its components
    \begin{equation}
        \label{eq:HochschildDef}
	\delta_\Total = -[\argument, \mu_\Total]_\Total
	\quad
        \textrm{ and }
        \quad
	\delta_\Wobs = -[\argument, \mu]_\Wobs,
    \end{equation}
    is a coisotropic chain map of degree 1 and $\delta^2 = 0$.
\end{proposition}
\begin{proof}
    Since $\mu_\Total$ is an associative multiplication on
    $\algebra{A}_\Total$ we know that
    $\delta_\Total \colon \mathrm{C}^\bullet(\algebra{A}_\Total)
    \longrightarrow \mathrm{C}^{\bullet+1}(\algebra{A}_\Total)$ is a
    differential.  Moreover, it is clear that
    $\delta_\Wobs \colon \mathrm{C}^\bullet(\algebra{A})_\Wobs
    \longrightarrow \mathrm{C}^\bullet(\algebra{A})_\Wobs$ is also a
    differential and it preserves the $\NULL$-component by the
    definition of $[\argument, \argument]_\Wobs$.  Finally, we have
    for $(\Phi_\Total,\Phi_\Wobs) \in \mathrm{C}^n(\algebra{A})_\Wobs$
    that
    $(\delta_\Total \circ \iota_n)(\Phi_\Total,\Phi_\Wobs) =
    \delta_\Total(\Phi_\Total) =
    \iota_{n+1}(\delta_\Wobs((\Phi_\Total,\Phi_\Wobs)))$ holds, and
    hence $(\delta_\Total,\delta_\Wobs)$ is a coisotropic morphism.
\end{proof}

The coisotropic Hochschild differential can be interpreted as twisting
the coisotropic DGLA
$(\mathrm{C}^\bullet(\algebra{A}), [\argument, \argument],0)$ with the
Maurer-Cartan element $\mu \in \mathrm{C}^2(\algebra{A})_\Wobs$, but
with signs chosen in such a way that it corresponds to the usual
Hochschild differential.  More explicitly we have the following
result.
\begin{corollary}
    \label{corollary:HochschildExplicit}%
    Let $(\algebra{A},\mu) \in \CoisoAlgTriple_\coisoField{k}$ be a
    coisotropic algebra.  Then the coisotropic Hochschild differential
    $\delta \colon \mathrm{C}^\bullet(\algebra{A}) \longrightarrow
    \mathrm{C}^{\bullet +1}(\algebra{A})$ is given by
    $\delta = (\delta^{\algebra{A}_\Total},
    (\delta^{\algebra{A}_\Total},\delta^{\algebra{A}_\Wobs}))$, where
    $\delta^{\algebra{A}_\Total}$ and $\delta^{\algebra{A}_\Wobs}$
    denote the Hochschild differentials of the algebras
    $(\algebra{A}_\Total, \mu_\Total)$ and
    $(\algebra{A}_\Wobs, \algebra{A}_\Wobs)$, respectively.
\end{corollary}

From this explicit characterization of the coisotropic Hochschild
differential in terms of the classical Hochschild differentials it
becomes clear that
$(\mathrm{C}^\bullet(\algebra{A}),[\argument,\argument],\delta)$ is a
coisotropic DGLA.

\begin{definition}[Coisotropic Hochschild complex]
    \label{definition:CoisoHochschildComplex}%
    Let $(\algebra{A},\mu) \in \CoisoAlgTriple_\coisoField{k}$ be a
    coisotropic algebra.  The coisotropic DGLA
    $(\mathrm{C}^\bullet(\algebra{A}),[\argument,\argument],\delta)$
    is called the \emph{(coisotropic) Hochschild complex} of
    $\algebra{A}$.
\end{definition}

Assigning the Hochschild complex to a given (coisotropic) algebra is
not functorial on all of $\CoisoAlgTriple_\coisoField{k}$.  But if we
restrict ourselves to the subcategory
$\CoisoAlgTriple_\coisoField{k}^\times$ of coisotropic algebras with
invertible morphisms we get a functor
$\mathrm{C}^\bullet \colon \CoisoAlgTriple_\coisoField{k}^\times \to
\CoisoDGLieAlgTriple$ by mapping each coisotropic algebra to its
Hochschild complex and every algebra isomorphism
$\phi \colon \algebra{A} \to \algebra{B}$ to
$\mathrm{C}^\bullet(\phi) \colon \mathrm{C}^\bullet(\algebra{A}) \to
\mathrm{C}^\bullet(\algebra{B})$ given by
$\mathrm{C}^\bullet(\phi)(f) = \phi \circ f \circ (\phi^{-1})^{\tensor
  n}$ for $f \in \mathrm{C}^n(\algebra{A})_{\Total/\Wobs}$.  A similar
construction clearly also works for usual algebras.  We can now show
that this functor commutes with reduction up to an injective natural
transformation.
\begin{proposition}[Hochschild complex vs. reduction]
    \label{prop:HochschildVSReduction}%
    There exists an injective natural transformation
    $\eta \colon \red \circ \mathrm{C}^\bullet \Longrightarrow
    \mathrm{C}^\bullet \circ \red$, i.e.
    \begin{equation}
	\begin{tikzcd}
            \CoisoAlgTriple_\coisoField{k}^\times
            \arrow{r}{\mathrm{C}^\bullet}
            \arrow{d}[swap]{\red}
            & \CoisoDGLieAlgTriple
            \arrow{d}{\red}
            \arrow[Rightarrow]{dl}[swap]{\eta}\\
            \Algebras_\field{k}^\times
            \arrow{r}{\mathrm{C}^\bullet}
            &\DGLieAlg
	\end{tikzcd}
    \end{equation}
    commutes.
\end{proposition}
\begin{proof}
    For every coisotropic algebra $\algebra{A}$ define
    $\eta_\algebra{A} \colon \mathrm{C}^\bullet(\algebra{A})_\red \to
    \mathrm{C}^\bullet(\algebra{A}_\red)$ by
    \begin{equation*}
        \eta_\algebra{A}([f])([a_1], \dotsc, [a_n])
        =
        \left[ f_\Wobs(a_1, \dotsc, a_n) \right].
    \end{equation*}
    for $[f]=[(f_\Total,f_\Wobs)] \in \mathrm{C}^n(\algebra{A})_\red$.
    First note that
    $\eta_\algebra{A}([f]) \colon \algebra{A}_\red^{\tensor n} \to
    \algebra{A}_\red$ is well defined since if $a_i \in \algebra{A}_0$
    for any $i = 1,\dotsc,n$ we have
    $f_\Wobs(a_1, \dotsc, a_n) \in \algebra{A}_\Null$ and hence
    $[f_\Wobs(a_1, \dotsc, a_n)] = 0$.  Moreover, $\eta_\algebra{A}$
    is well-defined since for $f \in \mathrm{C}^n(\algebra{A})_\Null$
    we have $f_\Wobs(a_1, \dotsc, a_n) \in \algebra{A}_\Null$ and thus
    $\eta([f]) = 0$.  To see that $\eta$ is indeed a natural
    transformation we need to show that for every isomorphism
    $\phi \colon \algebra{A} \to \algebra{B}$ we have
    $\eta_\algebra{B} \circ \mathrm{C}^\bullet(\phi)_\red =
    \mathrm{C}^\bullet([\phi]) \circ \eta_\algebra{A}$.  But it is
    clear after inserting the definitions.  Finally, suppose
    $\eta_\algebra{A}([f]) = \eta_\algebra{A}([g])$.  This means that
    $(f_\Wobs-g_\Wobs)(a_1, \dotsc, a_n) \in \algebra{A}_\Null$ and
    therefore $[f] = [g]$.  Thus $\eta_\algebra{A}$ is injective.
\end{proof}

Now let us turn to the cohomology of the Hochschild complex.
\begin{definition}[Coisotropic Hochschild cohomology]
    \label{definition:CoisoHCohom}%
    Let $(\algebra{A},\mu) \in \CoisoAlgTriple_\coisoField{k}$ be a
    coisotropic algebra.  The cohomology
    $\Hochschild^\bullet(\algebra{A}) = \ker \delta / \image \delta$
    of the Hochschild complex $\mathrm{C}^\bullet(A)$ is called the
    \emph{(coisotropic) Hochschild cohomology} of $\algebra{A}$.
\end{definition}
Using the definition of kernel, image and quotient in
$\CoisoModTriple_\coisoField{k}$ as given in
\autoref{sec:CoisotropicModules} \ref{Kernel}, \ref{CoisoImage},
\ref{Quotient} we can express the coisotropic Hochschild cohomology
more explicitly as follows.
\begin{lemma}
    \label{lem:CoisotropicHochschildCohomology}%
    The Hochschild cohomology of
    $\algebra{A} \in \CoisoAlgTriple_\coisoField{k}$ is given by
    \begin{align}
	\Hochschild^\bullet(\algebra{A})_\Total
	&=
        \Hochschild^\bullet(\algebra{A}_\Total),
        \\
	\Hochschild^\bullet(\algebra{A})_\Wobs
	&=
        \ker \delta_\Wobs / \image \delta_\Wobs,
        \\
        \shortintertext{and}
	\Hochschild^\bullet(\algebra{A})_\Null
	&=
        \ker (\delta_\Wobs\at{\Null}) / \image \delta_\Wobs
    \end{align}
    with
    \begin{align}
	\ker \delta_\Wobs^{n+1}
	&=
        \big\{
        (f_\Total, f_\Wobs) \in \mathrm{C}^{n+1}(\algebra{A})_\Wobs
	\; \big| \;
        \delta^{\algebra{A}_\Total} f_\Total = 0
	\textrm{ and }
        \delta^{\algebra{A}_\Wobs} f_\Wobs = 0
        \big\}
	\subseteq
        \ker \delta_{\algebra{A}_\Total}^{n+1}
        \times
        \ker \delta_{\algebra{A}_\Wobs}^{n+1},
        \\
	\image \delta_\Wobs^n
	&=
        \big\{
        (f_\Total, f_\Wobs) \in \mathrm{C}^{n+1}(\algebra{A})_\Wobs
	\; \big| \;
        \exists (g_\Total, g_\Wobs) \in \mathrm{C}^n(\algebra{A})_\Wobs:
	\delta^{\algebra{A}_\Total}g_\Total = f_\Total
	\textrm{ and }
	\delta^{\algebra{A}_\Wobs}g_\Wobs = f_\Wobs
        \big\},
        \\
        \shortintertext{and}
        \ker (\delta^n_\Wobs\at{\Null} )
	&=
        \big\{
        (f_\Total, f_\Wobs) \in \mathrm{C}^{n+1}(\algebra{A})_\Null
	\; \big| \;
        \delta^{\algebra{A}_\Total} f_\Total = 0
	\textrm{ and }
        \delta^{\algebra{A}_\Wobs} f_\Wobs = 0
        \big\}
	\subseteq
        \ker \delta_{\algebra{A}_\Total}^n
        \times
        \ker \delta_{\algebra{A}_\Wobs}^n.
	\end{align}
\end{lemma}
With this we can compute the zeroth and first Hochschild cohomology of
a given coisotropic algebra.  The following also shows that in low
degrees the interpretation of the coisotropic Hochschild cohomology is
analogous to that for usual algebras.
\begin{proposition}
    \label{prop:ZerothFirstHochschild}%
    Let $\algebra{A} \in \CoisoAlgTriple_\coisoField{k}$ be a
    coisotropic algebra.
    \begin{propositionlist}
    \item \label{prop:ZerothFirstHochschild_1}We have
        \begin{align}
            \Hochschild^0(\algebra{A})_\Total
            &=
            \Center(\algebra{A}_\Total), \\
            \Hochschild^0(\algebra{A})_\Wobs
            &=
            \big\{
            a \in \algebra{A}_\Wobs
            \; \big| \;
            a \in \Center(\algebra{A}_\Wobs)
            \textrm{ and }
            \iota_\algebra{A}(a) \in \Center(\algebra{A}_\Total)
            \big\},
            \\
            \shortintertext{and}
            \Hochschild^0(\algebra{A})_\Null
            &=
            \big\{
            a_0 \in \algebra{A}_\Null
            \; \big| \;
            a_0 \in \Center(\algebra{A}_\Wobs)
            \textrm{ and }
            \iota_\algebra{A}(a_0) \in \Center(\algebra{A}_\Total)
            \big\}.
	\end{align}
    \item \label{item:HHoneIsDer} We have
        \begin{align}
            \Hochschild^1(\algebra{A})_\Total
            &=
            \Der(\algebra{A}_\Total) / \InnDer(\algebra{A}_\Total),
            \\
            \Hochschild^1(\algebra{A})_\Wobs
            &=
            \Der(\algebra{A})_\Wobs
            /
            \big\{
            (D_\Total, D_\Wobs) \in \Der(\algebra{A})_\Wobs
            \; \big| \;
            \exists a \in \algebra{A}_\Wobs:
            D_\Total = [\argument, \iota_\algebra{A}(a)],
            D_\Wobs = [\argument, a]
            \big\},
            \\
            \shortintertext{and}
            \Hochschild^1(\algebra{A})_\Null
            &=
            \Der(\algebra{A})_\Null
            /
            \big\{
            (D_\Total, D_\Wobs) \in \Der(\algebra{A})_\Wobs
            \; \big| \;
            \exists a \in \algebra{A}_\Wobs:
            D_\Total = [\argument, \iota_\algebra{A}(a)],
            D_\Wobs = [\argument, a]
            \big\}.
	\end{align}
        Hence
        $\Hochschild^1(\algebra{A}) = \CoisoDer(\algebra{A}) /
        \CoisoInnDer(\algebra{A})$.
    \end{propositionlist}
\end{proposition}
\begin{proof}
    The first claim is clear by
    \autoref{lem:CoisotropicHochschildCohomology} and
    $\delta_{-1} = 0$. The second follows by carefully investigating
    the additional compatibility conditions beside the usual argument
    on derivations.
\end{proof}
\begin{remark}
    \label{remark:CenterAndOutsiders}%
    The \emph{center} of a coisotropic algebra could now be defined as
    $\Center(\algebra{A}) := \Hochschild^0(\algebra{A})$. Similarly,
    one can define the \emph{outer derivations} of a coisotropic
    algebra by $\OutDer(\algebra{A}) := \Hochschild^1(\algebra{A})$.
\end{remark}
\begin{remark}
    \label{remark:EnrichedStuff}%
    Using methods from enriched category theory one can define the
    center of a monoid $A$ internal to a given monoidal category
    $\category{C}$ by $[\category{C}, \category{C}](\id_A, \id_A)$.
    Here $[\category{C},\category{C}]$ denotes the
    $\category{C}$-enriched functor category of endofunctors of
    $\category{C}$.  Applying this to the monoidal category
    $\CoisoModTriple_\coisoField{k}$ yields exactly the notion of
    center of a coisotropic algebra as introduced in
    \autoref{remark:CenterAndOutsiders}.
\end{remark}
\begin{remark}
    \label{remark:HHandReduction}%
    Combining \autoref{prop:CohomologyCommutesWithReduction} with
    \autoref{prop:HochschildVSReduction} immediately shows that there
    exists an injective natural transformation
    $\eta \colon \red \circ \Hochschild^\bullet \Longrightarrow
    \Hochschild^\bullet \circ \red$.  In particular, for any
    coisotropic algebra $\algebra{A}$ we have
    \begin{equation}
        \label{eq:HHReduction}
        \Hochschild^\bullet(\algebra{A})_\red
        \subseteq
        \Hochschild^\bullet(\algebra{A}_\red).
    \end{equation}
\end{remark}
\begin{remark}
    \label{remark:HHOtherCoefficiants}%
    In this section we defined Hochschild cohomology of coisotropic
    algebras only with the algebra itself as coefficients.  It should
    be clear that all the above constructions also work for a
    coisotropic $\algebra{A}$-bimodule $\module{M}$ by using
    $\mathrm{C}^\bullet(\algebra{A},\module{M})$.
\end{remark}

%
%

\subsection{Formal Deformations}
\label{sec:FormalDeformations}

Throughout this section we will assume that the scalars satisfy
$\field{Q} \subseteq \field{k}$ in order to make use of the
description of deformations by Maurer-Cartan elements.

Let $(\algebra{A},\mu_0) \in \CoisoAlgTriple_\coisoField{k}$ be a
coisotropic associative $\field{k}$-algebra.  By
\autoref{def:DeformationOfCoisoAlgebra} a formal associative
deformation $(\algebra{A}\formal{\lambda},\mu)$ is given by an
associative multiplication
$\mu \colon \algebra{A}\formal{\lambda} \tensor
\algebra{A}\formal{\lambda} \longrightarrow
\algebra{A}\formal{\lambda}$ making $\algebra{A}\formal{\lambda}$ a
coisotropic $\field{k}\formal{\lambda}$-algebra such that
$\cl(\algebra{A},\mu)$ is given by $(\algebra{A},\mu_0)$, or in other
words
\begin{equation}
    \label{eq:MuIsFormalSeries}
    \mu = \mu_0 + \sum_{k=1}^{\infty} \lambda^k \mu_k
\end{equation}
with $\mu_k \colon \algebra{A} \tensor \algebra{A} \to \algebra{A}$.
Such deformations can now be understood as Maurer-Cartan elements in
the coisotropic DGLA
$\lambda \mathrm{C}^\bullet(\algebra{A})\formal{\lambda}$
corresponding to $(\algebra{A}\formal{\lambda},\mu_0)$.
\begin{lemma}
    \label{lemma:AssoDefIsMC}%
    Let $(\algebra{A},\mu) \in \CoisoAlgTriple_\coisoField{k}$ be a
    coisotropic associative $\field{k}$-algebra. A multiplication
    $\mu = \mu_0 + M$, with $M = \sum_{k=1}^{\infty} \lambda^k \mu_k$
    is a formal associative deformation of $\mu_0$ if and only if
    \begin{equation}
        \label{eq:MisMC}
        \delta M + \frac{1}{2}[M, M] = 0.
    \end{equation}
\end{lemma}
\begin{proof}
    This can be seen very easily. By
    \autoref{lem:AssoCoisotropicMultiplication} we know that we have
    to check that $[\mu_\Total, \mu_\Total]_{\algebra{A}_\Total} = 0$
    and $[\mu_\Wobs, \mu_\Wobs]_{\algebra{A}_\Wobs} = 0$.  Thus,
    consider the total component of $\mu$ as
    $\mu_\Total = (\mu_0)_\Total + M_\Total$.  We have
    \begin{equation*}
        [\mu_\Total, \mu_\Total]
        =
        [(\mu_0)_\Total + M_\Total, (\mu_0)_\Total + M_\Total]
        =
        2 \delta M_\Total + [M_\Total, M_\Total],
    \end{equation*}
    where we used the associativity of ${\mu_0}_\Total$ and the graded
    skew-symmetry of Gerstenhaber bracket. The very same holds for the
    $\WOBS$-component.
\end{proof}

Let us now consider two formal associative deformations $\mu$ and
$\mu'$ of $(\algebra{A},\mu_0)$.  We say that they are
\emph{equivalent} if there exists
$T = \id + \lambda(\ldots) \in
\CoisoAut_{\coisoField{k}\formal{\lambda}}(\algebra{A}\formal{\lambda})$
such that $T \circ \mu = \mu' \circ (T \tensor T)$, i.e. we have
\begin{equation}
    \label{eq:EquivalenceInComponents}
    T_\Total(\mu_\Total (a,b))
    =
    \mu'_\Total (T_\Total(a), T_\Total(b))
    \quad
    \textrm{and}
    \quad
    T_\Wobs(\mu_\Wobs (a,b))
    =
    \mu'_\Wobs (T_\Wobs(a), T_\Wobs(b))
\end{equation}
for $a, b \in \algebra{A}_{\Total/\Wobs}$.  Thus, as in the case of
associative algebras, there exists a unique
$D = \sum_{k=0}^\infty \lambda^k D_k \in
\CoisoHom_\coisoField{k}(\algebra{A}\formal{\lambda},
\algebra{A}\formal{\lambda})$ such that $T = \exp(\lambda D)$.
This allows us to conclude the following claim.
\begin{lemma}
    \label{lem:equivalentDefs}%
    Two formal associative deformations $\mu$ and $\mu'$ of a
    coisotropic $\field{k}$-algebra $(\algebra{A},\mu_0)$ are
    equivalent if and only if there exists
    $D \in \CoisoHom_\coisoField{k}(\algebra{A}\formal{\lambda},
    \algebra{A}\formal{\lambda})$ such that
    \begin{equation}
	\label{eq:equivalentDefs}
        \E^{\lambda \ad (D)}(\mu) =\mu',
    \end{equation}
    where $\ad(D) = [D,\argument]$ using the coisotropic Gerstenhaber bracket.
\end{lemma}
Note that \eqref{eq:equivalentDefs} is equivalent to
\begin{align}
    \E^{\lambda \ad_\Total (D_\Total)}(\mu_\Total)
    &=
    \mu'_\Total
    \\
    \E^{\lambda \ad_\Wobs (D_\Wobs)}(\mu_\Wobs)
    &=
    \mu'_\Wobs.
\end{align}
Summing up the above lemmas we can state the relation between formal
deformations and the deformation functor.
\begin{theorem}[Equivalence classes of deformations]
    \label{theorem:EquivalenceClassesOfDefs}%
    Let $\field{k}$ be a commutative ring with
    $\field{Q} \subseteq \field{k}$.  Let $(\algebra{A},\mu_0)$ be a
    coisotropic $\field{k}$-algebra.  Then the set of equivalence
    classes of formal associative deformations of $\algebra{A}$
    coincides with $\Def(\mathrm{C}^\bullet(\algebra{A}))$, where
    $\mathrm{C}^\bullet(\algebra{A})$ is the coisotropic Hochschild
    DGLA of $\algebra{A}$.
\end{theorem}
\begin{proof}
    By \autoref{lemma:AssoDefIsMC} we know that formal associative
    deformations of $\mu_0$ correspond to Maurer-Cartan elements of
    $\lambda \mathrm{C}^\bullet(\algebra{A})\formal{\lambda}$, while
    by \autoref{lem:equivalentDefs} two such deformations are
    equivalent if and only if they lie in the same
    $\group{G}(\mathrm{C}^\bullet(\algebra{A}))$-orbit.  Hence
    $\Def(\mathrm{C}^\bullet(\algebra{A}))$ is exactly the set of
    equivalence classes of formal deformations.
\end{proof}

Finally, we can reformulate the classical theorem about the extension
of a deformation up to a given order for coisotropic algebras.
\begin{theorem}[Obstructions]
    \label{theorem:Obstructions}%
    Let $\field{k}$ be a commutative ring with
    $\field{Q} \subseteq \field{k}$.  Let
    $(\algebra{A},\mu_0) \in \CoisoAlgTriple_\coisoField{k}$ be a
    coisotropic $\field{k}$-algebra.  Furthermore, let
    $\mu^{(k)} = \mu_0 + \cdots + \lambda^k \mu_k \in
    \mathrm{C}^2(\algebra{A})_\Wobs$ be an associative deformation of
    $\mu_0$ up to order $k$.  Then
    \begin{equation}
	\label{eq:thm:ExtendingDeformation}
	R_{k+1}
        =
        \left(
            \frac{1}{2} \sum_{\ell=1}^k
            \big[
            (\mu_\ell)_\Total,
            (\mu_{k+1-\ell})_\Total
            \big]^{\algebra{A}_\Total},
            \frac{1}{2} \sum_{\ell=1}^{k}
            \big[
            (\mu_\ell)_\Wobs,
            (\mu_{k+1-\ell})_\Wobs
            \big]^{\algebra{A}_\Wobs}
        \right)
	\in
        \mathrm{C}^3(\algebra{A})_\Wobs
    \end{equation}
    is a coisotropic Hochschild cocycle, i.e.
    $\delta_\Wobs R_{k+1} = 0$.  The deformation $\mu^{(k)}$ can be
    extended to order $k+1$ if and only if
    $R_{k+1}= \delta_\Wobs \mu_{k+1}$.  In this case every such
    $\mu_{k+1}$ yields an extension
    $\mu^{(k+1)} = \mu^{(k)} + \lambda^{k+1}\mu_{k+1}$.
\end{theorem}
\begin{proof}
    By the classical deformation theory of associative algebras it is
    clear that \eqref{eq:thm:ExtendingDeformation} is closed since
    $\delta_\Wobs =
    (\delta^{\algebra{A}_\Total},\delta^{\algebra{A}_\Wobs})$.  If
    $R_{k+1}$ is exact, we know that $\mu^{(k)}_\Total$ and
    $\mu^{(k)}_\Wobs$ can be extended via $(\mu_{k+1})_\Total$ and
    $(\mu_{k+1})_\Wobs$, respectively.  Thus $\mu_{k+1}$ yields an
    extension of $\mu^{(k)}$.  On the other hand, if $\mu^{(k)}$ can
    be extended, we know that
    $(R_{k+1})_\Total = \delta^{\algebra{A}_\Total}
    (\mu_{k+1})_\Total$ and
    $(R_{k+1})_\Wobs = \delta^{\algebra{A}_\Wobs} (\mu_{k+1})_\Wobs$.
    Hence, $R_{k+1} = \delta_\Wobs \mu_{k+1}$.
\end{proof}

Thus the obstructions for deformations of an associative structure on
a coisotropic module are given by $\Hochschild^3(\algebra{A})_\Wobs$.
The coisotropic module $\Hochschild^3(\algebra{A})$ carries more
information than just the obstructions to deformations of the
coisotropic algebra $\algebra{A}$.  Since
$\Hochschild^3(\algebra{A})_\Total =
\Hochschild^3(\algebra{A}_\Total)$ it also encodes the obstructions of
deformations of the classical algebra $\algebra{A}_\Total$.  Moreover,
$\Hochschild^3(\algebra{A})_\Null$ is important for the reduction of
$\Hochschild^3(\algebra{A})$ and hence controls which obstructions on
$\algebra{A}$ descend to obstructions on $\algebra{A}_\red$.  In
particular, we have seen at the end of
\autoref{sec:CoisotropicHochschildCohomology} that
$\Hochschild^3(\algebra{A})_\red \subseteq
\Hochschild^3(\algebra{A}_\red)$.

%
%

\subsection{Example I: BRST Reduction}
\label{subsec:ExampleIBRST}

The above definition of a deformation of a coisotropic algebra
recovers the following two interesting examples from deformation
quantization. The first comes from BRST reduction of star products.

We recall the situation of \cite{bordemann.herbig.waldmann:2000a,
  gutt.waldmann:2010a}. Consider a Poisson manifold $M$ with a
strongly Hamiltonian action of a connected Lie group $G$ and momentum
map $J\colon M \longrightarrow \liealg{g}^*$, where $\liealg{g}$ is
the Lie algebra of $G$. One assumes that the classical level surface
$C = J^{-1}(\{0\}) \subseteq M$ is a non-empty (necessarily
coisotropic) submanifold by requiring $0$ to be a regular value of
$J$. Moreover, we assume that the action of $C$ is proper. Then we
have the classical coisotropic triple
\begin{equation}
	\algebra{A} := (\Cinfty(M), \algebra{B}_C, \algebra{J}_C),
\end{equation} where
$\algebra{J}_C = \ker \iota^* \subseteq \Cinfty(M)$ is the vanishing
ideal of the constraint surface $C \subseteq M$ and $\algebra{B}_C$
its Poisson normalizer. Next, we assume to have a star product $\star$
invariant under the action of $G$ which admits a deformation
$\deform{J}$ of $J$ into a quantum momentum map. In the symplectic
case such star products always exist since we assume the action of $G$
to be proper, see \cite{reichert.waldmann:2016a} for a complete
classification and further references. In the general Poisson case the
situation is less clear.

Out of this a coisotropic $\mathbb{C}\formal{\lambda}$-algebra
$\qalgebra{A} := (\Cinfty(M)\formal{\lambda}, \qalgebra{B}_C,
\qalgebra{J}_C)$ is then constructed, where
$\qalgebra{J}_C = \ker \deform{\iota}^* \subseteq
\Cinfty(M)\formal{\lambda}$ is the quantum vanishing ideal given by
the kernel of the deformed restriction
$\deform{\iota}^* = \iota^* \circ S$.  Here
$S = \id + \sum_{k=1}^{\infty} \lambda^k S_k$ is a formal power series
of differential operators guaranteeing that $\qalgebra{J}_C$ is indeed
a left ideal with respect to $\star$.  In fact, $S$ can be chosen to
be $G$-invariant.

We now want to construct a coisotropic algebra
$(\Cinfty(M)\formal{\lambda}, \algebra{B}_C\formal{\lambda},
\algebra{J}_C\formal{\lambda})$ which is isomorphic to
$(\Cinfty(M)\formal{\lambda}, \qalgebra{B}_C, \qalgebra{J}_C)$.  For
this, note that
$S\colon \Cinfty(M)\formal{\lambda} \longrightarrow
\Cinfty(M)\formal{\lambda}$ is invertible, hence we get a star product
\begin{equation}
	f \star^S g = S(S^{-1}f \star S^{-1}g)
\end{equation}
on $\Cinfty(M)\formal{\lambda}$.  From
$\deform{\iota}^* = \iota^* \circ S$ directly follows, that $S$ maps
$\qalgebra{J}_C$ to $\algebra{J}_C\formal{\lambda}$. It is slightly
less evident, but follows from the characterization of the normalizer
$\qalgebra{B}_C$ as those functions whose restriction to $C$ are
$G$-invariant, that $S$ maps the normalizer $\qalgebra{B}_C$ to the
normalizer $\qalgebra{B}^S_C$ of $\algebra{J}_C\formal{\lambda}$ with
respect to $\star^S$.  Finally, we know that $f \in \qalgebra{B}_C$ if
and only if for all $\xi \in \liealg{g}$ it holds that
$0 = \Lie_{\xi_C} \deform{\iota}^* f = \Lie_{\xi_C} \iota^*Sf$.  Hence
$f \in \qalgebra{B}_C$ if and only if
$Sf \in \algebra{B}_C\formal{\lambda}$.  Thus $S$ is an isomorphism of
coisotropic triples
\begin{equation}
	((\Cinfty(M)\formal{\lambda},\star), \qalgebra{B}_C, \qalgebra{J}_C)
	\overset{S}{\longrightarrow}
	((\Cinfty(M)\formal{\lambda},\star^S), \algebra{B}_C\formal{\lambda},
	\algebra{J}_C\formal{\lambda}).
\end{equation}
In particular, we have a deformation
$\qalgebra{A}^S = ((\Cinfty(M)\formal{\lambda},\star^S),
\algebra{B}_C\formal{\lambda}, \algebra{J}_C\formal{\lambda})$ of the
classical coisotropic triple in this case, and the coisotropic triple
$\qalgebra{A}$ is isomorphic to it.

%
%

\subsection{Example II: Coisotropic Reduction in the Symplectic Case}
\label{subsec:ExampleIBRST}

While the previous example makes use of a Lie group symmetry, the
following relies on a coisotropic submanifold only. However, at the
present state, we have to restrict ourselves to a symplectic manifold
$(M, \omega)$. Thus let $\iota\colon C \longrightarrow M$ be a
coisotropic submanifold. We assume that the classical reduced phase
space $M_\red = C \big/ \mathord{\sim}$ is smooth with the projection
map $\pi\colon C \longrightarrow M_\red$ being a surjective
submersion. It follows that there is a unique symplectic form
$\omega_\red$ on $M_\red$ with $\pi^*\omega_\red = \iota^*\omega$. We
follow closely the construction of Bordemann in \cite{bordemann:2005a,
  bordemann:2004a:pre} to construct a deformation of the classical
coisotropic triple
$\algebra{A} = (\algebra{A}_\Total, \algebra{A}_\Wobs,
\algebra{A}_\Null)$ given by $\algebra{A}_\Total = \Cinfty(M)$ with
the vanishing ideal
$\algebra{A}_\Null = \algebra{J}_C \subseteq \Cinfty(M)$ of $C$ and
the Poisson normalizer
$\algebra{A}_\Wobs = \algebra{B}_C \subseteq \Cinfty(M)$ of
$\algebra{J}_C$ as before.

To this end, one considers the product $M \times M_\red^-$ with the
symplectic structure $\pr^*_M \omega -
\pr_{M_\red}^*\omega_\red$. Then
\begin{equation}
    \label{eq:CintoMMred}
    I\colon
    C \ni p
    \; \mapsto \;
    (\iota(p), \pi(p)) \in M \times M_\red
\end{equation}
is an embedding of $C$ as a Lagrangian submanifold. By Weinstein's
Lagrangian neighbourhood theorem one has a tubular neighbourhood
$U \subseteq M \times M_\red$ and an open neighbourhood $V \subseteq T^*C$
of the zero section $\iota_C\colon C \longrightarrow T^*C$ in the
cotangent bundle $\pi_C\colon T^*C \longrightarrow C$ with a
symplectomorphism $\Psi\colon U \longrightarrow V$, where $T^*C$ is
equipped with its canonical symplectic structure, such that
$\Psi \circ I = \iota_C$.

In the symplectic case, star products $\star$ are classified by their
characteristic or Fedosov class $c(\star)$ in
$\HdR^2(M, \field{C})[[\lambda]]$. The assumption of having a smooth
reduced phase space allows us now to choose star products $\star$ on
$M$ and $\star_\red$ on $M_\red$ in such a way that
$\iota^*c(\star\at{U}) = \pi^* c(\star_\red)$. Note that this is a
non-trivial condition on the relation between $\star$ and $\star_\red$
which, nevertheless, always has solutions. Given such a matching pair
we have a star product $\star \tensor \star_\red^\opp$ on
$M \times M_\red^-$ by taking the tensor product of the individual
ones. Note that we need to take the opposite star product on the
second factor as we also took the negative of $\omega_\red$ needed to
have a Lagrangian embedding in \eqref{eq:CintoMMred}. It follows that
the characteristic class
$c\big((\star \tensor \star_\red^\opp)\at{U}\big) = 0$ is trivial.

On the cotangent bundle $T^*C$ the choice of a covariant derivative
induces a standard-ordered star product $\star_\std$ together with a
left module structure on $\Cinfty(C)\formal{\lambda}$ via the
corresponding symbol calculus, see
\cite{bordemann.neumaier.waldmann:1998a}. The characteristic class of
$\star_\std$ is known to be trivial, $c(\star_\std) = 0$, see
\cite{bordemann.neumaier.pflaum.waldmann:2003a}. Hence the pull-back
star product $\Psi^*(\star_\std\at{V})$ is equivalent to
$(\star \tensor \star_\red^\opp)\at{U}$. Hence we find an equivalence
transformation between $\Psi^*(\star_\std)$ and
$\star \tensor \star_\red$ on the tubular neighbourhood $U$. Using
this, we can also pull-back the left module structure to obtain a left
module structure on $\Cinfty(C)\formal{\lambda}$ for the algebra
$\Cinfty(M \times M_\red)\formal{\lambda}$. Note that here we even get
an extension to all functions since the left module structure with
respect to $\star_\std$ coming from the symbol calculus is by
differential operators and $\Psi \circ I = \iota_C$. Hence the module
structure with respect to $\star \tensor \star_\red^\opp$ is by
differential operators as well. This ultimately induces a left module
structure $\acts$ on $\Cinfty(C)\formal{\lambda}$ with respect to
$\star$ and a right module structure $\racts$ with respect to
$\star_\red$ such that the two module structures commute: we have a
bimodule structure. Moreover, it is easy to see that the module
endomorphisms of the left $\star$-module are given by the right
multiplications with functions from $\Cinfty(M_\red)\formal{\lambda}$,
i.e.
\begin{equation}
    \label{eq:EndosCinftyC}
    \End_{(\Cinfty(M)\formal{\lambda}, \star)}
    (\Cinfty(C)\formal{\lambda})^\opp
    \cong
    \Cinfty(M_\red)\formal{\lambda}.
\end{equation}
Moreover, one can construct from the above equivalences a formal
series $S = \id + \sum_{r=1}^\infty \lambda^r S_r$ of differential
operators $S_r$ on $M$ such that the left module structure is given by
\begin{equation}
    \label{eq:LeftModuleOnC}
    f \acts \psi
    =
    \iota^*(S(f) \star \prol(\psi)),
\end{equation}
for $f \in \Cinfty(M)\formal{\lambda}$ and
$\psi \in \Cinfty(C)\formal{\lambda}$, where
$\prol\colon \Cinfty(C)\formal{\lambda} \longrightarrow
\Cinfty(M)\formal{\lambda}$ is the prolongation coming from the
tubular neighbourhood $U$.

The left module structure is cyclic with cyclic vector
$1 \in \Cinfty(C)\formal{\lambda}$. This means that
\begin{equation}
    \label{eq:KerModule}
    \qalgebra{J}_C
    =
    \left\{
        f \in \Cinfty(M)\formal{\lambda}
        \; \big| \;
        f \acts 1 = 0
    \right\}
\end{equation}
is a left $\star$-ideal and
$\Cinfty(C)\formal{\lambda} \cong \Cinfty(M)\formal{\lambda} \big/
\qalgebra{J}_C$ as left $\star$-modules. Moreover, the normalizer
\begin{equation}
    \label{eq:qBCDef}
    \qalgebra{B}_C = \normalizer_\star(\qalgebra{J}_C)
\end{equation}
with respect to $\star$ gives first
$\qalgebra{B}_C \big/ \qalgebra{J}_C \cong
\End_{(\Cinfty(M)\formal{\lambda}, \star)}
(\Cinfty(C)\formal{\lambda})^\opp$ for general reasons. Then this
yields the algebra isomorphism
$\qalgebra{B}_C \big/ \qalgebra{J}_C \cong
\Cinfty(M_\red)\formal{\lambda}$.

Thanks to the explicit formula for $\acts$ we can use the series $S$
to pass to a new equivalent star product $\star'$ such that
$\qalgebra{J}'_C = \algebra{J}_C\formal{\lambda}$.  We see that this
brings us precisely in the situation of
Example~\ref{example:DeformationCTriple}: The coisotropic algebra
$\qalgebra{A} = (\qalgebra{A}_\Total, \qalgebra{A}_\Wobs,
\qalgebra{A}_\Null)$ with
$\qalgebra{A}_\Total = (\Cinfty(M)\formal{\lambda}, \star)$ and
$\qalgebra{A}_\Wobs = \qalgebra{B}_C$ as well as
$\qalgebra{A}_\Null = \qalgebra{J}_C$ is isomorphic to a deformation
of the classical coisotropic algebra $\algebra{A}$ we started
with. Note that it might not be directly a deformation of
$\algebra{A}$ as we still might have to untwist first $\qalgebra{J}_C$
using $S$ and then $\qalgebra{B}_C$ as in
Example~\ref{example:DeformationCTriple}. This way we can give a
re-interpretation of Bordemann's construction in the language of
deformations of coisotropic algebras.

%
%

\subsection{Outlook}

When working with coisotropic algebras and related structures it is a
recurring theme to investigate the compatibility of a given
construction with the reduction functor.  We have seen in
\autoref{thm:defvsred} and \autoref{prop:HochschildVSReduction} that
the compatibility with reduction might only be given up to an
injective natural transformation, and in general it seems that one can
not expect much more.  Nevertheless, it would be rewarding to find
special situations in which the deformation functor $\Def$ or the
construction of the Hochschild complex commute with reduction up to a
natural \emph{iso}morphism.

Given a bimodule over a coisotropic algebra it should be clear that
one can define the coisotropic Hochschild complex and its cohomology
also with coefficients in the bimodule. This can then be used to
formulate also the deformation problem for (bi-)modules.

Having established coisotropic Hochschild cohomology and its
importance in deformation theory of coisotropic algebras one would
like to be able to actually compute it in certain cases.  A first
important example known from classical differential geometry is the
Hochschild-Kostant-Rosenberg theorem, implementing a bijection between
the Hochschild cohomology of the algebra of functions on a manifold
and its multivector fields.  A coisotropic version of this result for
coisotropic algebras of the form
$(\Cinfty(M), \Cinfty(M)^\mathcal{F}, \mathcal{J}_C)$, with $M$ a
smooth manifold, $\mathcal{J}_C$ the vanishing ideal of a submanifold
and $\Cinfty(M)^\mathcal{F}$ the functions on $M$ which are constant
along a foliation $\mathcal{F}$ on $C$, would be desirable.  To
achieve this it will be necessary to carry over other notions of
differential geometry, like multivector fields etc., to the
coisotropic setting. It will be important to consider also
geometrically motivated bimodules for the coefficients in such
scenarios. The cohomologies computed in
\cite{bordemann.et.al:2005a:pre} should be related to the coisotropic
Hochschild cohomology, at least for particular and simple cases of
submanifolds and foliations.

%
%

%
%

\end{document}